%% file: sgdNA.tex
\documentclass{article}

\PassOptionsToPackage{numbers, compress}{natbib}

\usepackage[preprint]{neurips_2020}

\usepackage[utf8]{inputenc} 
\usepackage[T1]{fontenc}    
\usepackage{hyperref}       
\usepackage{url}            
\usepackage{booktabs}       
\usepackage{amsfonts}       
\usepackage{nicefrac}       
\usepackage{microtype}      

\input{header}

\title{Debiasing Averaged Stochastic Gradient Descent to handle missing values}

\author{%
   Aude Sportisse$^{(1)}$,
   Claire Boyer$^{(1,2)}$,
   Aymeric Dieuleveut$^{(3)}$,
   Julie Josse$^{(3,4)}$ \\
   $^{(1)}$ Sorbonne Université, $^{(2)}$ ENS Paris, $^{(3)}$ Ecole Polytechnique, $^{(4)}$ Google France
}

\begin{document}

\maketitle

\input{main}

\appendix
\renewcommand{\thesection}{S\arabic{section}}
\renewcommand{\thetable}{S\arabic{table}}
\renewcommand{\thefigure}{S\arabic{figure}}
\renewcommand{\thetheorem}{S\arabic{theorem}}
\renewcommand{\thefootnote}{S\arabic{footnote}}

\vskip 0.3in

\section{Discussion on the paper of \citet{ma2017stochastic}}
\label{app:MN}

In this section, we make the theoretical issues unlocked in \citet{ma2017stochastic} explicit.  For clarity, we directly refer to the lemmas and theorems as numbered in the published version  (\url{http://www.global-sci.org/uploads/online_news/NMTMA/201809051633-2442.pdf}), the numbering being slightly different than the arXiv version.
For readability, we translate their method and results with the notation used in the present paper.
In their paper, they consider the finite-sample setting, with at hand $(D_k,\tilde{X}_i)_{1\leq i\leq n}$, in view of minimizing the empirical risk.

As a preamble, let us remind that the contributions of the present paper go far beyond correcting the approach in \citep{ma2017stochastic}: we propose a different algorithm using averaging, that converges faster and in a non-strongly convex regime, with a different proof technique, requiring a more technical proof on the second order moment of the noise, and we allow for heterogeneity in the missing data mechanism.

\subsection{Hurdles to get unbiased gradients of the empirical risk} 
The stochastic gradients in \cite{ma2017stochastic} are not unbiased gradients of the empirical risk (which makes their main result wrong). Indeed, their algorithm uses the debiased direction \eqref{eq:gradMCARgeneral} by sampling uniformly with replacement the $(\tilde{X_k})_k$'s.

For clarity, we highlight both why the result is not technically correct in their paper, and why it is not intuitively possible to achieve the result they give.

   \paragraph{Technically.} The proof of the main Theorem 2.2 (Theorem 2.1 being a direct corollary),
    corresponds to the classical proof in which one upper bounds the  expectation of the mean-squared distance from the iterate at iteration $k+1$ to the optimal point \textbf{conditionally to the iterate at iteration $k$}, or more precisely, conditionally to a $\sigma$-algebra making this iterate measurable. This is typically written $$\mathbb E\left[||\beta_{k+1}-\beta_*^n||^2 |\mathcal F_k\right],$$ where $\beta_*^n$ is the minimizer of the empirical risk $R_n$ and $\beta_{k}$ is $\mathcal F_k$-measurable.
    
    The crux of the proof is then to use \textbf{unbiased gradients conditionally to $\mathcal F_k$}: the property needed is that
    $$\mathbb E\left[ g_{i_{k+1}}(\beta_k) |\mathcal F_k\right] = \nabla R_n (\beta_k).$$
    
    In classical ERM (without missing value) it is done by sampling uniformly at iteration $k+1$ one observation indexed by $i_{k+1}\sim \mathcal U \llbracket 1; n\rrbracket$, \textbf{independently from $\beta_k$}.
     
    In regression with missing data, one has to deal with another source of randomness, the randomness of the mask $D$. In \citet{ma2017stochastic}, Lemma A.1  states that for a random $i\sim \mathcal U \llbracket 1; n\rrbracket$ and a matrix row $A_i$, for a random mask $D$ associated to this row, 
    $$\mathbb E_D [\mathbb E_i g_i(\beta)] = \nabla R_n(\beta).$$
    This lemma is valid.  Unfortunately, its usage in the proof of Theorem 2.2 (page 18, line (ii)), is not, as one \emph{does not  have}:
     $$\mathbb E [ g_{i_{k+1}}(\beta_k) | \mathcal{F}_k] = \nabla R_n(\beta_k),$$
    indeed, 
    \begin{itemize}
        \item either the sample $i_{k+1}$ is chosen uniformly at random in $ \llbracket 1; n\rrbracket$ and $D_{i_{k+1}}$ \textbf{is not} independent from $\beta_k$.
        \item or the sample  $i$ is not chosen uniformly in $ \llbracket 1; n\rrbracket$ (for example without replacement, as we do) and then the gradient is not an unbiased gradient of $R_n$ as the sampling is not uniform anymore.
    \end{itemize}
    
    In other words, the proof would only be valid if \textbf{the mask for the missing entries was re-sampled each time the point is used}, which is of course not realistic for a missing data approach (that would mean that the data has in fact been collected without missing entries). 
    
    \paragraph{Intuition on why it is hard.} A way to understand the impossibility of having a bound for multiple pass on ERM in the context of missing data is to underline that the empirical risk, in the presence of missing data, is an \textbf{unknown function}: its value cannot be computed exactly (see \Cref{rem:ERM}).
    
    As a consequence we can hardly expect that one could minimize it to unlimited accuracy.
    This is very similar to the situation for  the \emph{generalization risk} in a situation \emph{without missing data}: as the function is not observed, it is impossible to minimize it exactly.  Given only $n$ observations, no algorithm can achieve 0-generalization error (and statistical lower bounds \cite{tsybakov2003optimal} prove so).

    \paragraph{Conclusion.} This highlights how difficult it is to be rigorous when dealing with multiple sources of randomness. 
    Unfortunately, none of these limits are discussed in the current version of \citep{ma2017stochastic}.
    This makes the approach and the main theorem of \citep{ma2017stochastic} mathematically invalid. 
    In the present paper, the \emph{generalization risk} is decaying \emph{during the first pass}, and as a consequence, the empirical risk also probably does, but this has not been proved yet.

    In the following paragraph, we give details on the missing technical Lemma.
\subsection{Missing key Lemma in the proof.}  Proving that $(\tilde{f}_k)$ is a.s.\ convex is an important step for convergence, which was missing in the analysis of \cite{ma2017stochastic}. More precisely, in Lemma A.4. in  \cite{ma2017stochastic}, a condition is missing on $G(x)$: $G$ needs to be smooth  \emph{and convex} for its gradient  to satisfy the co-coercivity inequality. Note that this condition was also missing in the paper they refer to \citet{needell2014stochastic} (Indeed, at the third line of the proof of Lemma A.1.  in \citet{needell2014stochastic}, one needs $f$ to be convex for $G$ to be convex). Co-coercivity of the gradient is indeed a characterization of the fact that the function is smooth and convex, see for example \citet{Zhu_Mar_1995}.

\section{Proofs of technical lemmas}\label{supsec:proofs}

\setcounter{equation}{1}
Recall that we aim at minimizing the theoretical risk in both streaming and finite-sample settings. 
\begin{equation}
\label{sup-pb:online2}
\coeff^\star 
= \argmin_{\coeff \in \mathbb{R}^\dimvar} R(\beta) =  \argmin_{\coeff \in \mathbb{R}^\dimvar} \mathbb{E}_{(\design_{i:},y_i)}\left[f_i(\coeff)\right]. 
\end{equation}
\setcounter{theorem}{0}
In the sequel, one consider the following modified gradient direction
\setcounter{equation}{3}
\begin{equation}
\label{sup-eq:gradMCARgeneral}
\tilde{g}_{k}(\coeff_k)=P^{-1}\missdesignzero_{{k}:}\left(\missdesignzero_{{k}:}^TP^{-1}\coeff_k-y_{k}\right)-(\mathrm{I}-P)P^{-2}\mathrm{diag}\left(\missdesignzero_{{k}:}\missdesignzero_{{k}:}^T\right)\coeff_{k}.
\end{equation}

Note that for all $k$, $D_{k:}$ is independent from $(X_{k:},y_{k})$. In what follows, the proofs are derived considering
\begin{equation*}
\mathbb{E}=\mathbb{E}_{(X_{k:},y_{k}),D_{k:}}=
\mathbb{E}_{(X_{k:},y_{k})}\mathbb{E}_{D_{k:}}
\end{equation*}
where $\mathbb{E}_{(X_{k:},y_k)}$ and $\mathbb{E}_{D_{k:}}$   denotes the expectation with respect to the distribution of $(X_{k:},y_k)$ and $D_{k:}$ respectively.

\subsection{Proof of ~\cref{lem:unbiased}}\label{supsec:lem:unbiased}
\begin{lemma} \label{lem:unbiased-app}Let $(\mathcal{F}_k)_{k\geq 0}$ be the following $\sigma$-algebra, 
	 	\begin{equation*}
\mathcal{F}_k=  \sigma (X_{1:},y_1,D_{1:}\dots,X_{k:},y_k,D_{k:}).
	 \end{equation*}
\label{ass:iterationtheo-ap-app} 
The modified gradient $\tilde{g}_{k}(\beta_{k-1})$ in \Cref{sup-eq:gradMCARgeneral} is $\mathcal{F}_{k}$-measurable and
	\begin{equation*}
		\mathbb{E}\left[\tilde{g}_{k}(\coeff_{k-1})\left.\right|\mathcal{F}_{k-1}\right]= \nabla R(\coeff_{k-1}) \quad  \textrm{a.s.}  
	\end{equation*}
\end{lemma}

\begin{proof}
	\begin{align*}
	\mathbb{E}_{(X_{k:},y_k),D_{k:}}\left[\tilde{g}_k(\coeff_{k-1}) | \mathcal{F}_{k-1}\right]\overset{(i)}{=}& \:  \mathbb{E}_{(X_{k:},y_k),D_{k:}}\left[P^{-1}\missdesignzero_{k:}\missdesignzero_{k:}^TP^{-1}\right]\coeff_{k-1}-\mathbb{E}_{(X_{k:},y_k),D_{k:}}\left[P^{-1}\missdesignzero_{k:}y_k\right] \\
	&-\mathbb{E}_{(X_{k:},y_k),D_{k:}}\left[(\mathrm{I}-P)P^{-2}\mathrm{diag}\left(\missdesignzero_{k:}\missdesignzero_{k:}^T\right)\right]\coeff_{k-1} \\
	\overset{(ii)}{=}&\: \mathbb{E}_{(X_{k:},y_{k})}\left[P^{-1}P\design_{k:}\design_{k:}^TPP^{-1}\coeff_{k-1}-P^{-2}(P-P^2)\mathrm{diag}(\design_{k:}\design_{k:}^T)\coeff_{k-1}-P^{-1}P\design_{k:}y_k\right] \\
	& -\mathbb{E}_{(X_{k:},y_{k})}\left[(\mathrm{I}-P)P^{-2}P\mathrm{diag}\left(\design_{k:}\design_{k:}^T\right)\coeff_{k-1}\right] \\ 
	=& \: \nabla R(\coeff_{k-1}),
	\end{align*}
	In step (i), we use that $\beta_{k-1}$ is $\mathcal{F}_{k-1}$-measurable and $(X_k,y_k,D_k)$ is independent from $\mathcal{F}_{k-1}$. Step (ii) follows from 
	$$
	\left\{
	\begin{array}{cl}
	\mathbb{E}_{D_{k:}}\left[\tilde{X}_{k:}\tilde{X}_{k:^T} \right] &=PX_{k:}X_{k:}^TP-(P-P^2)\mathrm{diag}(X_{k:}X_{k:}^T), \\ 
	\mathbb{E}_{D_{k:}}\left[\mathrm{diag}(\tilde{X}_{k:}\tilde{X}_{k:}^T)\right] & =P\mathrm{diag}({X}_{k:}{X}_{k:}^T), \\
	\mathbb{E}_{D_{k:}}\left[\tilde{X}_{k:}\right]&=P{X}_{k:}.
	\end{array}
	\right.
	$$

\end{proof}

\subsection{Proof of ~\cref{lem:noisetheo}}\label{appsec:lem:noisetheo}
\begin{lemma}
\label{lem:noisetheo-app} 
 The additive noise process $(\tilde{g}_{k}(\coeff^\star))_k$ with $\beta^\star$ defined in \Cref{sup-pb:online2} is $\mathcal{F}_{k}-$measurable and has the following properties:
	\begin{enumerate}
		\item \label{ass:noisezeroexp} $\forall k\geq 0, \: \mathbb{E}[\tilde{g}_{k}(\coeff^\star) \left.\right| \mathcal{F}_{k-1}]=0$ a.s.,
		\item \label{ass:noisebound} $ \forall k\geq 0, \: \mathbb{E}[\|\tilde{g}_{k}(\coeff^\star)\|^2 \left.\right| \mathcal{F}_{k-1}]$ is a.s. finite,
		\item \label{ass:noisecov} $\forall k\geq 0, \: \mathbb{E}[\tilde{g}_{k}(\coeff^\star)\tilde{g}_{k}(\coeff^\star)^T] \preccurlyeq C(\beta^\star)=c(\beta^\star)H,$ where $\preccurlyeq$ denotes the order between self-adjoint operators ($A\preccurlyeq B$ if $B-A$ is positive semi-definite).
	\end{enumerate}
\end{lemma}

\begin{proof}
	
	\underline{\ref{ass:noisezeroexp}} The first point is easily verified using Lemma \ref{ass:iterationtheo} combined with $\nabla R(\coeff^{\star})=0$ by \eqref{sup-pb:online2}.
	\medbreak
	\underline{\ref{ass:noisebound}} 
	Let us first remark that by independence $\mathbb{E}[\|\tilde{\gradRbis}(\coeff^\star)\|^2 \left.\right| \mathcal{F}_{k-1}]=\mathbb{E}[\|\tilde{\gradRbis}(\coeff^\star)\|^2]$. Then,
	$$\mathbb{E}[\|\tilde{\gradRbis}(\coeff^\star)\|^2]\leq \frac{1}{p_m^2}\mathbb{E}\left[\|\design_{k:}\|^2\left(\missdesignzero_{k:}^TP^{-1}\coeff^\star-y_k\right)^2 \right]+\frac{(1-p_m)^2}{p_m^2}\mathbb{E}\left[\lVert P^{-1}\mathrm{diag}\left(\missdesignzero_{k:}\missdesignzero_{k:}^T\right)\coeff^{\star}\rVert^2\right].$$
	We decompose the computation with respect to $\mathbb{E}_{D_{k:}}$ first,
	\begin{align*}
	\mathbb{E}_{D_{k:}}\left[\left(\missdesignzero_{k:}^TP^{-1}\coeff^\star-y_k\right)^2\right]&=\mathbb{E}_{D_{k:}}\left[(\missdesignzero_{k:}^TP^{-1}\coeff^\star)^2\right]-2y_k\mathbb{E}_{D_{k:}}\left[\missdesignzero_{k:}^TP^{-1}\coeff^\star\right]+\pred_k^2 \\
	&=\mathbb{E}_{D_{k:}}\left[\left(\sum_{j=1}^{d} \missdesignzero_{kj}p_j^{-1}\coeff_{ j}^{\star}\right)^2\right]-2\pred_k\mathbb{E}_{D_{k:}}\left[\sum_{j=1}^{d}\missdesignzero_{kj}p_j^{-1}\coeff_{j}^\star\right]+\pred_k^2 \\
	&=\sum_{j=1}^{d}\mathbb{E}_{D_{k:}}\left[\missdesignzero_{kj}^2p_j^{-2}\coeff_{j}^{\star 2}\right]+2\sum_{l<j} \mathbb{E}_{D_{k:}}\left[\missdesignzero_{kj}\missdesignzero_{kl}p_l^{-1}\coeff_{j}^\star\coeff_{l}^\star\right]-2\pred_k \sum_{j=1}^{d}X_{kj}\coeff_{j}^\star+\pred_k^2 \\
	&=\sum_{j=1}^{d}p_j^{-1}X_{kj}^2\coeff_{ j}^{\star 2}+2\sum_{l<j}X_{kj}X_{kl}\coeff_{j}^\star\coeff_{l}^\star - 2\pred_k\sum_{j=1}^{d}X_{kj}\coeff_{j}^\star+\pred_k^2 \\
	&=(\design_{k:}^T\coeff^\star-\pred_k)^2+\sum_{j=1}^{d}(p_j^{-1}-1)X_{kj}^2\coeff_{j}^{\star 2}, 
	\end{align*}
	which gives 
	\setcounter{equation}{7 }
	\begin{equation}\label{eq:proofT1exp}
	\mathbb{E}_{D_{k:}}\left[\left(\missdesignzero_{k:}^TP^{-1}\coeff^\star-y_k\right)^2\right]\leq (X_{k:}^T\coeff^\star-\pred_k)^2+\frac{1-p_m}{p_m}\coeff^{\star T}\textrm{diag}(X_{k:}X_{k:}^T)\coeff^\star.
	\end{equation}
	As for the second term,
	\begin{align*}
	\mathbb{E}_{D_{k:}}\left[\|P^{-1}\textrm{diag}(\missdesignzero_{k:}\missdesignzero_{k:}^T)\coeff^\star\|^2\right]&=\mathbb{E}_{D_{k:}}\left[\sum_{j=1}^{d}\missdesignzero_{kj}^4p_j^{-2}\coeff_{ j}^{\star 2}\right] \\
	&=\sum_{j=1}^{d}X_{kj}^4p_j^{-1}\coeff_{j}^{\star 2} \\
	&\leq \frac{1}{p_m}\sum_{j=1}^{d} X_{ij}^4\coeff_{j}^{\star 2} \\
	&\leq \frac{1}{p_m}\left(\sum_{j=1}^{d} X_{kj}^2\right)\left(\sum_{j=1}^{d} X_{kj}^2\coeff_{j}^{\star 2}\right) \\
	&=\frac{1}{p_m}\|X_{k:}\|^2\coeff^{\star T}\textrm{diag}(X_{k:}X_{k:}^T)\coeff^{\star}
	\end{align*}
	Finally, one obtains
	$$\mathbb{E}[\|\tilde{\gradRbis}(\coeff^\star)\|^2 \left.\right| \mathcal{F}_{k-1}]\leq\frac{1}{p_m^2}\mathbb{E}_{(X_{k:},y_k)}\left[(\epsilon_k)^2\|X_{k:}\|^2\right]+\frac{(1-p_m)+(1-p_m)^2}{p_m^3}\mathbb{E}_{(X_{k:},y_k)}\left[\|X_{k:}\|^2\coeff^{\star T}\textrm{diag}(X_{k:}X_{k:}^T)\coeff^\star\right].$$ 
	
	\underline{\ref{ass:noisecov}} 
	We aim at proving there exists $H$ such that $$\mathbb{E}[\tilde{\gradRbis}(\coeff^\star)\tilde{\gradRbis}(\coeff^\star)^T] \preccurlyeq C=cH.$$
	Simple computations lead to: 
	$$\mathbb{E}[\tilde{g}_k(\coeff^\star)\tilde{g}_k(\coeff^\star)^T]=\mathbb{E}[T_1+T_2+T_2^T+T_3], $$
	with:
	\begin{align*}
	T_1&=(\tilde{\design}_{k:}^TP^{-1}\coeff^\star-\pred_k)^2P^{-1}\tilde{\design}_{k:}\tilde{\design}_{k:}^T P^{-1},\\	
	T_{2}&=-(\tilde{\design}_{k:}^TP^{-1}\coeff^\star-\pred_k)P^{-1}\tilde{\design}_{k:}\coeff^{\star T}\mathrm{diag}(\tilde{\design}_{k:}\tilde{\design}_{k:}^T)P^{-2}(I-P),\\
	T_{3}&=(I-P)P^{-2}\mathrm{diag}(\tilde{\design}_{k:}\tilde{\design}_{k:}^T)\coeff^\star \coeff^{\star T} \mathrm{diag}(\tilde{\design}_{k:}\tilde{\design}_{k:}^T)P^{-2}(I-P).
	\end{align*}
	
	\paragraph{Bound on $T_1$.}
	
	For the first term, we use 
	\begin{equation}\label{eq:proofboundP}
	P^{-1}\tilde{\design}_{k:}\tilde{\design}_{k:}^T P^{-1}\preccurlyeq \frac{1}{p_m^2}\tilde{\design}_{k:}\tilde{\design}_{k:}^T,
	\end{equation}
	since for all vector $v\neq 0, \: v^T\left(\frac{1}{p_m^2}\tilde{\design}_{k:}\tilde{\design}_{k:}^T - P^{-1}\tilde{\design}_{k:}\tilde{\design}_{k:}^T P^{-1}\right)v \geq 0$,
	\begin{align*}
	&\sum_{j=1}^{d} \left(\frac{1}{p_m^2}-\frac{1}{p_j^2}\right)\tilde{\design}_{kj}^2
	v_j^2+2\sum_{1\leq j<l\leq d } \left(\frac{1}{p_m^2}-\frac{1}{p_j p_l}\right)\tilde{\design}_{kj}\tilde{\design}_{kl}v_jv_l \\
	&\overset{(iii)}{\geq} \sum_{j=1}^{d} \left(\frac{1}{p_m^2}-\frac{1}{p_j^2}\right)\tilde{\design}_{kj}^2
	v_j^2 +2\sum_{1\leq j<l\leq d } \sqrt{\left(\frac{1}{p_m^2}-\frac{1}{p_j^2 }\right)\left(\frac{1}{p_m^2}-\frac{1}{p_l^2 }\right)}\tilde{\design}_{kj}\tilde{\design}_{kl}v_jv_l \\
	&=\left(\sum_{j=1}^{d}\sqrt{\left(\frac{1}{p_m^2}-\frac{1}{p_j^2}\right)}\tilde{\design}_{kj}v_j\right)^2 \geq 0.
	\end{align*}
	Step (iii) uses $\left(\frac{1}{p_m^2}-\frac{1}{p_j p_l}\right)\geq \sqrt{\left(\frac{1}{p_m^2}-\frac{1}{p_j^2 }\right)\left(\frac{1}{p_m^2}-\frac{1}{p_l^2 }\right)}.$ Indeed, \begin{align*}
	\left(\frac{1}{p_m^2}-\frac{1}{p_j p_l}\right)^2\geq \left(\frac{1}{p_m^2}-\frac{1}{p_j^2 }\right)\left(\frac{1}{p_m^2}-\frac{1}{p_l^2 }\right) &\Leftrightarrow \left(\frac{1}{p_m^4}-2\frac{1}{p_j p_l}\frac{1}{p_m^2}+\frac{1}{p_j^2 p_l^2}\right)-\frac{1}{p_m^4}+\frac{1}{p_m^2p_l^2}+\frac{1}{p_m^2p_j^2}-\frac{1}{p_j^2p_l^2}\geq 0 \\
	&\Leftrightarrow \left(\frac{1}{p_mp_j}-\frac{1}{p_mp_l}\right)^2\geq0.
	\end{align*}
    
	Let us now prove that
	$$\frac{1}{p_m^2}\tilde{\design}_{k:}\tilde{\design}_{k:}^T\preccurlyeq\frac{1}{p_m^2}\design_{k:}\design_{k:}^T$$
    i.e.\
	\begin{equation}\label{eq:proofboundT1}
	\tilde{\design}_{k:}\tilde{\design}_{k:}^T\preccurlyeq\design_{k:}\design_{k:}^T.
	\end{equation}
	Indeed, for all vector $v\neq 0$, $v^T(\design_{k:}\design_{k:}^T-\tilde{\design}_{k:}\tilde{\design}_{k:}^T)v\geq 0$:
	\begin{align*}
	v^T(\design_{k:}\design_{k:}^T-\tilde{\design}_{k:}\tilde{\design}_{k:}^T)v&=\sum_{j=1}^{\dimvar}(1-\delta_{kj}^2)\design_{kj}^2v_j^2+2\sum_{1\leq j<l\leq \dimvar}(1-\delta_{kj}\delta_{kl})\design_{kj}\design_{kl}v_jv_l \\
	&\overset{(iv)}{\geq} \sum_{j=1}^{\dimvar}(1-\delta_{kj}^2)\design_{kj}^2v_j^2+2\sum_{1\leq j<l\leq \dimvar}\sqrt{(1-\delta_{kj}^2)(1-\delta_{kl}^2)}\design_{kj}\design_{kl}v_jv_l\\ &=\left(\sum_{j=1}^{\dimvar}\sqrt{(1-\delta_{kj}^2}\design_{kj}v_j\right)^2 \geq 0
	\end{align*}
	Step (iv) is obtained using $(1-\delta_{kj}\delta_{kl})\geq \sqrt{(1-\delta_{kj}^2)(1-\delta_{kl}^2)}$. Indeed, 
	\begin{align*}
	(1-\delta_{kj}\delta_{kl})^2\geq (1-\delta_{kj}^2)(1-\delta_{kl}^2) &\Leftrightarrow (1-2\delta_{kl}\delta_{kj}+\delta_{kj}^2\delta_{kl}^2)-1+\delta_{kj}^2-\delta_{kj}^2\delta_{kl}^2+\delta_{kl}^2\geq 0 \\
	&\Leftrightarrow (\delta_{kj}-\delta_{kl})^2\geq0.
	\end{align*}

	Then, by \eqref{eq:proofT1exp} and $(\design_{k:}^T\coeff^\star-\pred_k)^2=\epsilon_k^2$, 
	\begin{align*}
	\mathbb{E}_{(X_{k:},y_k)}\left[T_1\right]&=\mathbb{E}_{(X_{k:},y_k)}\left[\frac{1}{p_m^2}\epsilon_k^2\design_{k:}\design_{k:}^T\right]+\mathbb{E}_{(X_{k:},y_k)}\left[\frac{1-p_m}{p_m^3}\left(\coeff^{\star T}\textrm{diag}(\design_{k:}\design_{k:}^T)\coeff^{\star}\right)\design_{k:}\design_{k:}^T\right]. \\
	\end{align*}
	Noting that
	\begin{equation}
	\label{eq:proofhessian_T1}
	\|\mathrm{diag}(\design_{k:})\coeff^{\star}\|^2\leq \|\design_{k:}\|^2\|\coeff^{\star}\|^2,
	\end{equation}
	
	\begin{equation}\label{eq:boundT1}
	\mathbb{E}\left[T_1\right]\preccurlyeq\frac{1}{ p_m^2}\mathrm{Var}(\epsilon_k)H+\frac{1-p_m}{ p_m^3}\|X_{k:}\|^2\|\coeff^\star\|^2H
	\end{equation}
	
	\paragraph{Bound on $T_3$.}
	
	Using the resulting matrix structure of $(I-P)P^{-2}\mathrm{diag}(\tilde{\design}_{k:}\tilde{\design}_{k:}^T)\coeff^\star \coeff^{\star T} \mathrm{diag}(\tilde{\design}_{k:}\tilde{\design}_{k:}^T)P^{-2}(I-P)$,
	\begin{equation*}
	\begin{pmatrix}
	\left(\coeff^{\star}_{1}\right)^2\delta_{k1}^4\design_{k1}^4 & 	\coeff^{\star}_{1}\coeff^{\star}_{2}\delta_{k1}^2\delta_{k2}^2\design_{k1}^2\design_{k2}^2 & \\
	& \ddots & \\
	& & \left(\coeff^{\star}_{\dimvar}\right)^2\delta_{k\dimvar}^4\design_{k\dimvar}^4
	\end{pmatrix},
	\end{equation*}
	one obtains
	\begin{multline}\label{eq:proofgeneralT3}
	\mathbb{E}_{D_{k:}}\left[T_3\right]
	=\underbrace{(I-P)P^{-2}P\mathrm{diag}({\design_{k:}}{\design_{k:}}^T)\coeff^\star \coeff^{\star T} \mathrm{diag}({\design_{k:}}{\design_{k:}}^T)PP^{-2}(I-P)}_{=:T_{3a}}
	\\
	+\underbrace{(I-P)P^{-2}(P-P^2)\mathrm{diag}({\design_{k:}}{\design_{k:}}^T)\mathrm{diag}(\coeff^\star \coeff^{\star T}) \mathrm{diag}({\design_{k:}}{\design_{k:}}^T)P^{-2}(I-P)}_{=:T_{3b}}.
	\end{multline}

	Using similar arguments as in \eqref{eq:proofboundP}, both terms in \eqref{eq:proofgeneralT3} are bounded as follows
	\begin{align*}
	T_{3a} &\preccurlyeq \frac{(1-p_m)^2}{p_m^2}\mathrm{diag}({\design_{k:}}{\design_{k:}}^T)\coeff^\star \coeff^{\star T} \mathrm{diag}({\design_{k:}}{\design_{k:}}^T) \\
	T_{3b} &\preccurlyeq\frac{(1-p_m)^3}{p_m^3}\mathrm{diag}({\design_{k:}}{\design_{k:}}^T)\mathrm{diag}(\coeff^\star \coeff^{\star T}) \mathrm{diag}({\design_{k:}}{\design_{k:}}^T)
	\end{align*}

	For $T_{3a}$, one can go further by using
	\begin{equation}\label{eq:prooft31t21prime}
	\mathrm{diag}(\design_{k:}\design_{k:}^T)\coeff^\star \coeff^{\star T}\mathrm{diag}(\design_{k:}\design_{k:}^T) \preccurlyeq \|\mathrm{diag}(\design_{k:})\coeff^{\star}\|^2\design_{k:}\design_{k:}^T.
	\end{equation}
	Let us prove that for all vector $v\neq 0$, $$v^T(\|\mathrm{diag}(\design_{k:})\coeff^{\star}\|^2\design_{k:}\design_{k:}^T-\mathrm{diag}(\design_{k:}\design_{k:}^T)\coeff^\star \coeff^{\star T}\mathrm{diag}(\design_{k:}\design_{k:}^T) )v\geq0, \: \textrm{i.e.}$$
	$$\underbrace{\sum_{j=1}^{d}\left(\left(\sum_{l=1}^{d}\design_{il}^2\coeff_{\star l}^2\right)\design_{kj}^2-\design_{kj}^4\coeff_{ j}^{\star 2}\right)v_j^2+2\sum_{1\leq j<m\leq \dimvar}\left(\left(\sum_{l=1}^{d}\design_{kl}^2\coeff_{ l}^{\star 2}\right)\design_{kj}\design_{km}-\coeff_{j}^\star\coeff_{ m}^\star\design_{kj}^2\design_{km}^2\right)v_mv_j}_{=:Q}\geq0$$
	
	Indeed, $Q\geq\left(\sum_{j=1}^{d}\sqrt{\left(\sum_{l=1}^{d}\design_{kl}^2\coeff^{\star 2}_ l\right)\design_{kj}^2-\design_{kj}^4\coeff_{j}^{\star 2}}v_j\right)^2\geq0$, since, looking at the term depending only on $v_j v_m$:
	
\begin{footnotesize}
	\begin{align*}
	&\left(\left(\sum_{l=1}^{d}\design_{kl}^2\coeff_{ l}^{\star 2}\right)\design_{kj}\design_{km}-\coeff_{j}^\star\coeff_{m}^\star\design_{kj}^2\design_{km}^2\right) \geq \sqrt{\left(\left(\sum_{l=1}^{d}\design_{kl}^2\coeff_{l}^{\star 2}\right)\design_{kj}^2-\design_{kj}^4\coeff_{j}^{\star 2}\right)\left(\left(\sum_{l=1}^{d}\design_{kl}^2\coeff_{ l}^{\star 2}\right)\design_{km}^2-\design_{km}^4\coeff_{m}^{\star 2}\right)} \\
	&\Leftrightarrow \left(\sum_{l=1}^{d}\design_{kl}^2\coeff_{ l}^{\star 2}\right)\design_{kj}^4\design_{km}^2\coeff_{j}^{\star 2} +\left(\sum_{l=1}^{d}\design_{kl}^2\coeff_{ l}^{\star 2}\right)\design_{km}^4\design_{kj}^2\coeff_{ m}^{\star 2}-2\left(\sum_{l=1}^{d}\design_{kl}^2\coeff_{ l}^{\star 2}\right)\design_{kj}^3\design_{km}^3\coeff_{j}^\star\coeff_{m}^\star\geq0 \\
	&\Leftrightarrow \left(\sqrt{\left(\sum_{l=1}^{d}\design_{kl}^2\coeff_{l}^{\star 2}\right)}\design_{kj}^2\design_{km}\coeff_{j}^\star - \sqrt{\left(\sum_{l=1}^{d}\design_{kl}^2\coeff_{ l}^{\star 2}\right)}\design_{km}^2\design_{kj}\coeff_{m}^\star\right)^2\geq 0
	\end{align*}
\end{footnotesize}
	
	For $T_{3b}$, one can also dig deeper noting that
	\begin{equation}\label{eq:prooft32t22prime}
	\mathrm{diag}(\design_{k:}\design_{k:}^T)\mathrm{diag}(\coeff^\star \coeff^{\star T})\mathrm{diag}(\design_{k:}\design_{k:}^T) \preccurlyeq  \|\mathrm{diag}(\design_{k:})\coeff^{\star}\|^2\design_{k:}\design_{k:}^T.
	\end{equation}
	
	For all vector $v\neq 0$, we aim at proving
	\begin{align*}
	&v^T(\|\coeff^{\star T}\mathrm{diag}(X_{k:})\|^2\design_{k:}\design_{k:}^T-\mathrm{diag}(\design_{k:}\design_{k:}^T)\mathrm{diag}(\coeff^\star \coeff^{\star T})  \mathrm{diag}(\design_{k:}\design_{k:}^T)))v\geq 0 \\
	&\Leftrightarrow \underbrace{\sum_{j=1}^{d}\left(\left(\sum_{l=1}^{d}\design_{kl}^2\coeff_{l}^{\star 2} \right)\design_{kj}^2-\design_{kj}^4\coeff_{ j}^{\star 2}\right)v_j^2+2\sum_{1\leq j<m\leq \dimvar}\left(\sum_{l=1}^{d}\design_{kl}^2\coeff_{ l}^{\star 2}\right)\design_{kj}\design_{km}v_jv_m }_{=: Q'} \geq 0.
	\end{align*}
	Indeed, $Q'\geq\left(\sum_{j=1}^{d}\sqrt{\left(\sum_{l=1}^{d}\design_{kl}^2\coeff_{ l}^{\star 2}\right)\design_{kj}^2-\design_{kj}^4\coeff_{ j}^{\star 2}}v_j\right)^2\geq0$ since
	
	\begin{align*}
	&\left(\left(\sum_{l=1}^{d}\design_{kl}^2\coeff_{ l}^{\star 2}\right)\design_{kj}\design_{km}\right) \geq \sqrt{\left(\left(\sum_{l=1}^{d}\design_{kl}^2\coeff_{l}^{\coeff 2}\right)\design_{kj}^2-\design_{kj}^4\coeff_{ j}^{\star 2}\right)\left(\left(\sum_{l=1}^{d}\design_{kl}^2\coeff_{ l}^{\star 2}\right)\design_{km}^2-\design_{km}^4\coeff_{m}^{\star 2}\right)}\\
	&\Leftrightarrow \left(\sum_{l=1}^{d}\design_{kl}^2\coeff_{ l}^{\star 2}\right)\design_{kj}^4\design_{km}^2\coeff_{j}^{\star 2} +\left(\sum_{l=1}^{d}\design_{kl}^2\coeff_{ l}^{\star 2}\right)\design_{km}^4\design_{kj}^2\coeff_{ m}^{\star 2}-\design_{kj}^4\design_{km}^4\coeff_{j}^{\star 2}\coeff_{m}^{\star 2}\geq 0 
	\end{align*}
	
	Combining \eqref{eq:proofhessian_T1},  \eqref{eq:prooft31t21prime} and \eqref{eq:prooft32t22prime} lead to 
	\begin{align*}
	\mathbb{E}_{(X_{k:},y_k)}\left[T_{3a}\right]&\preccurlyeq \frac{(1-p_m)^2}{p_m^2}\|X_{k:}\|^2\|\coeff^\star\|^2H \\
	\mathbb{E}_{(X_{k:},y_k)}\left[T_{3b}\right]&\preccurlyeq\frac{(1-p_m)^3}{p_m^3}\|X_{k:}\|^2\|\coeff^\star\|^2H 
	\end{align*}
	and to the final bound for $T_3$,
	\begin{equation}\label{eq:boundT3}
	    \mathbb{E}\left[T_3\right] \preccurlyeq \frac{(1-p_m)^2}{p_m^2}\|X_{k:}\|^2\|\coeff^\star\|^2H+\frac{(1-p_m)^3}{p_m^3}\|X_{k:}\|^2\|\coeff^\star\|^2H.
	\end{equation}
	
	\paragraph{Bound on $T_{2}+T_{2}^T$.}
    Firstly, focus on $T_2$: 

	\begin{align*}
	T_{2}&=-(\tilde{\design}_{k:}^TP^{-1}\coeff^\star-\pred_k)P^{-1}\tilde{\design}_{k:}\coeff^{\star T}\mathrm{diag}(\tilde{\design}_{k:}\tilde{\design}_{k:}^T)P^{-2}(I-P) \\
	&=:-(A-B),
	\end{align*}
	where  \begin{align*}
	A&=P^{-1}\tilde{\design}_{k:}\tilde{\design}_{k:}^TP^{-1}\coeff^\star\coeff^{\star T}\mathrm{diag}(\tilde{\design}_{k:}\tilde{\design}_{k:}^T)P^{-2}(I-P) \\
	B&=P^{-1}\tilde{\design}_{k:}\pred_k\coeff^{\star T}\mathrm{diag}(\tilde{\design}_{k:}\tilde{\design}_{k:}^T)P^{-2}(I-P).
	\end{align*}
	
	\subparagraph{Computation w.r.t.\ $\mathbb{E}_{D_{k:}}$.}
	
	Term $A$ can be split into three terms, 
	\begin{align*}
	A_1&=P^{-1}\mathrm{diag}(\tilde{\design}_{k:}\tilde{\design}_{k:}^T)P^{-1}\coeff^\star\coeff^{\star T}\mathrm{diag}(\tilde{\design}_{k:}\tilde{\design}_{k:}^T)P^{-2}(I-P) \\
	A_2&=P^{-1}(\tilde{\design}_{k:}\tilde{\design}_{k:}^T-\mathrm{diag}(\tilde{\design}_{k:}\tilde{\design}_{k:}^T))P^{-1}\mathrm{diag}(\coeff^\star\coeff^{\star T})\mathrm{diag}(\tilde{\design}_{k:}\tilde{\design}_{k:}^T)P^{-2}(I-P) \\
	A_3&=P^{-1}(\tilde{\design}_{k:}\tilde{\design}_{k:}^T-\mathrm{diag}(\tilde{\design}_{k:}\tilde{\design}_{k:}^T))P^{-1}(\coeff^\star\coeff^{\star T}-\mathrm{diag}(\coeff^\star\coeff^{\star T}))\mathrm{diag}(\tilde{\design}_{k:}\tilde{\design}_{k:}^T)P^{-2}(I-P).
	\end{align*}
	Noting that $$A_1=P^{-2}\mathrm{diag}(\tilde{\design}_{k:}\tilde{\design}_{k:}^T)\coeff^\star\coeff^{\star T}\mathrm{diag}(\tilde{\design}_{k:}\tilde{\design}_{k:}^T)P^{-2}(I-P),$$ the expectation $\mathbb{E}_{D_{k:}}$ has already been computed in \eqref{eq:proofgeneralT3}, so
	\begin{multline}\label{eq:proofboundA1}
	\mathbb{E}_{D_{k:}}\left[A_1\right]=P^{-2}P\mathrm{diag}({\design_{k:}}{\design_{k:}}^T)\coeff^\star \coeff^{\star T} \mathrm{diag}({\design_{k:}}{\design_{k:}}^T)PP^{-2}(I-P) \\
	+P^{-2}(P-P^2)\mathrm{diag}({\design_{k:}}{\design_{k:}}^T)\mathrm{diag}(\coeff^\star \coeff^{\star T}) \mathrm{diag}({\design_{k:}}{\design_{k:}}^T)P^{-2}(I-P).
	\end{multline}
	As for $A_2$, making the structure of the term $(\tilde{\design}_{k:}\tilde{\design}_{k:}^T-\mathrm{diag}(\tilde{\design}_{k:}\tilde{\design}_{k:}^T))P^{-1}\mathrm{diag}(\coeff^\star\coeff^{\star T})\mathrm{diag}(\tilde{\design}_{k:}\tilde{\design}_{k:}^T)$ explicit,
	$$A_2 = \begin{pmatrix}
	0 & \frac{1}{p_2}\delta_{k1}\delta_{k2}^3\design_{k1}\design_{k2}^3\coeff_2^{\star 2} & \dots & \frac{1}{p_d}\delta_{k1}\delta_{kd}^3\design_{k1}\design_{kd}^3\coeff_d^{\star2} \\
	\frac{1}{p_1}\delta_{k2}\delta_{k1}^3\design_{k2}\design_{k1}^3\coeff_1^{\star2} & 0 & & \\
	& & \ddots & \\
	\frac{1}{p_1}\delta_{kd}\delta_{k1}^3\design_{kd}\design_{k1}^3\coeff_1^{\star 2} & & & 0
	\end{pmatrix},$$
	one has 
	\begin{equation}\label{eq:proofA2bound}
	\mathbb{E}_{D_{k:}}[A_2]=P^{-1}P({\design_{k:}}{\design_{k:}}^T-\mathrm{diag}({\design_{k:}}{\design_{k:}}^T))\mathrm{diag}(\coeff^\star \coeff^{\star T})\mathrm{diag}({\design_{k:}}{\design_{k:}}^T)P^{-2}(I-P).
	\end{equation}
	
	As for $A_3$, the term $(\tilde{\design}_{k:}\tilde{\design}_{k:}^T-\mathrm{diag}(\tilde{\design}_{k:}\tilde{\design}_{k:}^T))P^{-1}(\coeff^\star\coeff^{\star T}-\mathrm{diag}(\coeff^\star\coeff^{\star T}))\mathrm{diag}(\tilde{\design}_{k:}\tilde{\design}_{k:}^T)$ can be made explicit as
	$$\begin{pmatrix}
	\sum_{l=2}^{d}\frac{1}{p_l}\delta_{kl}\design_{kl}\coeff_{l}^\star\delta_{k1}\design_{k1}^3\coeff_1^\star & \sum_{l=3}^{d}\frac{1}{p_l}\delta_{kl}\design_{kl}\coeff_{l}^\star\delta_{k1}\delta_{k2}^2\design_{k1}\design_{k2}^2\coeff_2^\star & \dots & \sum_{l\neq1,d}\frac{1}{p_l}\delta_{kl}\design_{kl}\coeff_{l}^\star\delta_{k1}\delta_{kd}^2\design_{k1}\design_{kd}^2\coeff_d^\star \\
	& \ddots & & \\
	& & \ddots & \\
	& & & \sum_{l=1}^{d-1}\frac{1}{p_l}\delta_{kl}\design_{kl}\coeff_{l}^\star\delta_{kd}\design_{kd}^3\coeff_d^\star
	\end{pmatrix},$$
	which gives
	\begin{multline*}
	\mathbb{E}_{D_{k:}}\left[A_3\right]=P^{-1}P(\tilde{\design}_{k:}\tilde{\design}_{k:}^T-\mathrm{diag}(\tilde{\design}_{k:}\tilde{\design}_{k:}^T))(\coeff^\star\coeff^{\star T}-\mathrm{diag}(\coeff^\star\coeff^{\star T}))\mathrm{diag}(\tilde{\design}_{k:}\tilde{\design}_{k:}^T)PP^{-2}(I-P) \\
	+P^{-1}(P-P^2)\mathrm{diag}\left((\tilde{\design}_{k:}\tilde{\design}_{k:}^T-\mathrm{diag}(\tilde{\design}_{k:}\tilde{\design}_{k:}^T))(\coeff^\star\coeff^{\star T}-\mathrm{diag}(\coeff^\star\coeff^{\star T}))\mathrm{diag}(\tilde{\design}_{k:}\tilde{\design}_{k:}^T)\right)P^{-2}(I-P).
	\end{multline*}
	
	Noting the following,
	\begin{multline*}
	\mathrm{diag}\left((\tilde{\design}_{k:}\tilde{\design}_{k:}^T-\mathrm{diag}(\tilde{\design}_{k:}\tilde{\design}_{k:}^T))(\coeff^\star\coeff^{\star T}-\mathrm{diag}(\coeff^\star\coeff^{\star T}))\mathrm{diag}(\tilde{\design}_{k:}\tilde{\design}_{k:}^T)\right)=\mathrm{diag}\left(\tilde{\design}_{k:}\tilde{\design}_{k:}^T\coeff^\star\coeff^{\star T}\mathrm{diag}(\tilde{\design}_{k:}\tilde{\design}_{k:}^T)\right)\\
	-\mathrm{diag}(\tilde{\design}_{k:}\tilde{\design}_{k:}^T)\mathrm{diag}(\coeff^\star\coeff^{\star T})\mathrm{diag}(\tilde{\design}_{k:}\tilde{\design}_{k:}^T), 
	\end{multline*}
	one has 
    \begin{multline}\label{eq:proofA3bound}
	\mathbb{E}_{D_{k:}}\left[A_3\right]=P^{-1}P(\tilde{\design}_{k:}\tilde{\design}_{k:}^T-\mathrm{diag}(\tilde{\design}_{k:}\tilde{\design}_{k:}^T))(\coeff^\star\coeff^{\star T}-\mathrm{diag}(\coeff^\star\coeff^{\star T}))\mathrm{diag}(\tilde{\design}_{k:}\tilde{\design}_{k:}^T)PP^{-2}(I-P) \\
	+P^{-1}(P-P^2)\mathrm{diag}\left(\tilde{\design}_{k:}\tilde{\design}_{k:}^T\coeff^\star\coeff^{\star T}\mathrm{diag}(\tilde{\design}_{k:}\tilde{\design}_{k:}^T)\right)P^{-2}(I-P) \\ 
	-P^{-1}(P-P^2)\mathrm{diag}(\tilde{\design}_{k:}\tilde{\design}_{k:}^T)\mathrm{diag}(\coeff^\star\coeff^{\star T})\mathrm{diag}(\tilde{\design}_{k:}\tilde{\design}_{k:}^T)P^{-2}(I-P)
	\end{multline}
	
	Term $B$ can be made explicit as follows
	\begin{equation*}
	\missdesignzero_{k:}\coeff^{\star T}\mathrm{diag}(\tilde{\design}_{k:}\tilde{\design}_{k:}^T)=\begin{pmatrix}
	\coeff_{1}^\star\delta_{i1}^3\design_{i1}^3 & \coeff_{1}^\star\delta_{i1}^2\delta_{i2}\design_{i1}^2\design_{i2} & \\
	\coeff_{2}^\star\delta_{i2}^2\design_{i2}^2 \delta_{i1}\design_{i1} & \coeff_{2}^\star\delta_{i2}^3\design_{i2}^3 & \\
	& & \ddots
	\end{pmatrix}
	\end{equation*}
	which implies
	\begin{multline}\label{eq:proofBbound}
	\mathbb{E}_{D_{k:}}\left[B\right]=\pred_k P^{-1}P\design_{k:}\coeff^{\star T}\mathrm{diag}({\design_{k:}}{\design_{k:}}^T)PP^{-2}(I-P)
	\\
	+\pred_kP^{-1}(P-P^2)\mathrm{diag}(\design_{k:}\coeff^{\star T}\mathrm{diag}({\design_{k:}}{\design_{k:}}^T))P^{-2}(I-P).
	\end{multline}
	
	Putting Equations \eqref{eq:proofboundA1}, \eqref{eq:proofA2bound}, \eqref{eq:proofA3bound} and \eqref{eq:proofBbound} together,  
	$$\mathbb{E}\left[T_{2}+T_{2}^T\right]=\mathbb{E}_{(X_{k:},y_{k})}\left[T_{21}+T_{22}+T_{23}+T_{23}^T+T_{24}+T_{24}^T+T_{25}\right]$$
	\begin{align*}
	T_{21}&=-2(P^{-1}-I)\mathrm{diag}({\design_{k:}}{\design_{k:}}^T)\coeff^\star\coeff^{\star T}\mathrm{diag}({\design_{k:}}{\design_{k:}}^T)(P^{-1}-I)\\
	T_{22}&=-2P^{-3}((I-P)(I-3P+2P^2)\mathrm{diag}({\design_{k:}}{\design_{k:}}^T)\mathrm{diag}(\coeff^\star\coeff^{\star T})\mathrm{diag}({\design_{k:}}{\design_{k:}}^T) \\
	T_{23}&=-{\design_{k:}}{\design_{k:}}^T\mathrm{diag}(\coeff^\star\coeff^{\star T})\mathrm{diag}({\design_{k:}}{\design_{k:}}^T)(P^{-2}(I-P)-P^{-1}(I-P))\\
	T_{24}&=-({\design_{k:}}^T\coeff^\star-\pred_k){\design_{k:}}\coeff^{\star T}\mathrm{diag}({\design_{k:}}{\design_{k:}}^T)P^{-1}(I-P) \\
	T_{25}&=-2({\design_{k:}}^T\coeff^\star-\pred_k)(I-P)\textrm{diag}({\design_{k:}}\coeff^{\star T}\mathrm{diag}({\design_{k:}}{\design_{k:}}^T))P^{-2}(I-P),
	\end{align*}
	
	\subparagraph{Computation w.r.t.\ $\mathbb{E}_{(X_{k:},y_{k})}$.}
	For $T_{21}$, it trivially holds that
	\begin{equation}\label{eq:proofTprime21}
	-\mathrm{diag}(\design_{k:}\design_{k:}^T)\coeff^\star \coeff^{\star T}\mathrm{diag}(\design_{k:}\design_{k:}^T) \preccurlyeq 0.
	\end{equation}
	Indeed, for all vector $v \neq 0$,
	$$\sum_{j=1}^{d} \design_{kj}^4\coeff_j^{\star 2}v_j^2+2\sum_{1\leq j<m\leq \dimvar}\coeff_j^\star\coeff_m^\star\design_{kj}^2\design_{km}^2v_jv_m=\left(\sum_{j=1}^{d}\design_{kj}^2\coeff_j^\star v_j\right)^2\geq 0.$$
	Denoting the maximum of the coefficients of $P$ as $p_M=\max_j p_j$, one has
	\begin{align*}
	T_{21}&\preccurlyeq
	-2\frac{(1-p_M)^2}{p_m^2}\mathrm{diag}({\design_{k:}}{\design_{k:}}^T)\coeff^\star\coeff^{\star T}\mathrm{diag}({\design_{k:}}{\design_{k:}}^T) &\\
	&\preccurlyeq 0 &\textrm{(using \eqref{eq:proofTprime21})}.
	\end{align*}
	
	$T_{22}$ is split into two terms, 
	\begin{align*}
	    T_{22a}&=-2P^{-3}((I-P)(I+2P^2))\mathrm{diag}({\design_{k:}}{\design_{k:}}^T)\mathrm{diag}(\coeff^\star\coeff^{\star T})\mathrm{diag}({\design_{k:}}{\design_{k:}}^T) \\
	    T_{22b}&=6P^{-2}(I-P)\mathrm{diag}({\design_{k:}}{\design_{k:}}^T)\mathrm{diag}(\coeff^\star\coeff^{\star T})\mathrm{diag}({\design_{k:}}{\design_{k:}}^T)
	\end{align*}
	$$T_{22a}\preccurlyeq -2\frac{(1-p_M)(1+2p_M^2)}{p_m^{3}}\mathrm{diag}({\design_{k:}}{\design_{k:}}^T)\mathrm{diag}(\coeff^\star\coeff^{\star T})\mathrm{diag}({\design_{k:}}{\design_{k:}}^T)
	\preccurlyeq 0,$$
	since it is a diagonal matrix with only negative coefficients, and noting that  $\frac{(1-p_M)(1+2p_M^2)}{p_m^{3}}>0$. 
	Then, 
	$$T_{22b} \preccurlyeq \frac{6(1-p_m)}{p_m^2}\mathrm{diag}({\design_{k:}}{\design_{k:}}^T)\mathrm{diag}(\coeff^\star\coeff^{\star T})\mathrm{diag}({\design_{k:}}{\design_{k:}}^T)$$
	which implies
	$$\mathbb{E}_{(X_{k:},y_{k})}\left[T_{22b}\right]\preccurlyeq \frac{6(1-p_m)}{p_m^2}\|\design_{k:}\|^2\|\coeff^{\star}\|^2H$$
	using \eqref{eq:prooft31t21prime} and \eqref{eq:proofhessian_T1}.
	
	As for $T_{23}+T_{23}^T$, note that
	$$T_{23}+T_{23}^T \preccurlyeq -2\frac{(p_M-1)^2}{p_m^2}\left({\design_{k:}}{\design_{k:}}^T\mathrm{diag}(\coeff^\star\coeff^{\star T})\mathrm{diag}({\design_{k:}}{\design_{k:}}^T)+\mathrm{diag}(\coeff^\star\coeff^{\star T})\mathrm{diag}({\design_{k:}}{\design_{k:}}^T){\design_{k:}}{\design_{k:}}^T\right) $$
	One prove that 
	\begin{multline}\label{eq:proofT23prime}
	-\left({\design_{k:}}{\design_{k:}}^T\mathrm{diag}(\coeff^\star\coeff^{\star T})\mathrm{diag}({\design_{k:}}{\design_{k:}}^T)+\mathrm{diag}({\design_{k:}}{\design_{k:}}^T)\mathrm{diag}(\coeff^\star\coeff^{\star T}){\design_{k:}}{\design_{k:}}^T\right) \\
	\preccurlyeq -2\left(\min_{j=1,\dots,\dimvar} \beta_{j}^{\star 2}\design_{kj}^2\right) \design_{k:}\design_{k:}^T
	\end{multline}
	
	Indeed, denoting $m=\left(\min_{j=1,\dots,\dimvar} \beta_{j}^{\star 2}\design_{kj}^2\right)$, one has
	
	\begin{align*}
	&v^T\left(-2m\design_{k:}\design_{k:}^T+\left({\design_{k:}}{\design_{k:}}^T\mathrm{diag}(\coeff^\star\coeff^{\star T})\mathrm{diag}({\design_{k:}}{\design_{k:}}^T)+\mathrm{diag}({\design_{k:}}{\design_{k:}}^T)\mathrm{diag}(\coeff^\star\coeff^{\star T}){\design_{k:}}{\design_{k:}}^T\right)\right)v\geq 0  \\
	\Leftrightarrow&\sum_{j=1}^\dimvar \left(-2m\design_{kj}^2+2\design_{kj}^4\coeff_{j}^{\star 2}\right)v_j^2+2\sum_{1\leq j<q\leq \dimvar} \left(-2m\design_{kj}\design_{kq}+\design_{kj}^3\design_{kq}\beta_{j}^{\star 2}+\design_{kq}^3\design_{kj}\beta_{q}^{\star 2}\right)v_jv_q \geq 0 \\
	\Leftrightarrow&\sum_{j=1}^\dimvar \left(-2m\design_{kj}^2+2\design_{kj}^4\coeff_{j}^{\star 2}\right)v_j^2+2\sum_{1\leq j<q\leq \dimvar} \sqrt{\left(-2m\design_{kj}^2+2\design_{kj}^4\coeff_{ j}^{\star 2}\right)\left(-2m\design_{kq}^2+2\design_{kq}^4\coeff_{q}^{\star 2}\right)}v_jv_q \geq 0 \\
	\Leftrightarrow& \left(\sum_{j=1}^\dimvar \sqrt{\left(-2m\design_{kj}^2+2\design_{kj}^4\coeff_{j}^{\star 2}\right)}v_j\right)^2 \geq 0,
	\end{align*}
	
	using that 
	\begin{align*}
	&\left(-2m\design_{kj}^2+2\design_{kj}^4\coeff_{j}^{\star 2}\right)\left(-2m\design_{kq}^2+2\design_{kq}^4\coeff_{q}^{\star 2}\right) \geq \left(-2m\design_{kj}\design_{kq}+\design_{kj}^3\design_{kq}\beta_{j}^{\star 2}+\design_{kq}^3\design_{kj}\beta_{q}^{\star 2}\right)^2 \\
	\Leftrightarrow& \left(\design_{kj}^3\design_{kq}\beta_{ j}^{\star 2}-\design_{kq}^3\design_{kj}\beta_{q}^{\star 2}\right)^2 \geq 0
	\end{align*}

	Therefore 
	$$\mathbb{E}_{(X_{k:},y_k)}\left[T_{23}+T_{23}^T\right]\preccurlyeq -2\frac{(p_M-1)^2}{p_m^2}\left(\min_{j=1,\dots,\dimvar} \beta_{j}^{\star 2}\design_{kj}^2\right) H \preccurlyeq 0, $$
	since $H$ is definite positive.

	Finally one uses $(X_{k:}^T\coeff^{\star}-\pred_k)=\epsilon_k$ to conclude by independence that
	$T_{24}=T_{25}=0$.
	
	One gets
	\begin{equation}\label{eq:boundT21T22}
	   \mathbb{E}\left[T_{2}+T_{2}^T\right] \preccurlyeq \frac{6(1-p_m)}{p_m^2}\|\design_{k:}\|^2\|\coeff^{\star}\|^2H.
	\end{equation}
	
	Combining \eqref{eq:boundT1}, \eqref{eq:boundT3} and \eqref{eq:boundT21T22} leads to the desired bound.

\end{proof}

\subsection{Proof of~\cref{lem:covtheo}}\label{appsec:lem:covtheo-app}
\begin{lemma}
	\label{sup-ass:covtheo-app}  For all $k \geq 0$, given the binary mask $D$, the adjusted gradient $\tilde{g}_k(\coeff)$ is a.s. $L_{k,D}$-Lipschitz continuous, i.e.\ for all $u, v \in \mathbb {R}^d$,
	\begin{equation*}
	\|\tilde{g}_{k}(u)-\tilde{g}_{k}(v)\|\leq L_{k,D}\|u-v\|~\textrm{a.s.}.
	\end{equation*}
	Set
	\begin{equation}
	\label{eq:smooth_const}
	L:=\sup_{k,D} L_{k,D} \leq \frac{1}{p_m^2} \max_k \|\design_{k:}\|^2~\textrm{a.s.}.
	\end{equation}
	In addition, for all $k \geq 0$, $\tilde{g}_{k}(\beta)$ is almost surely co-coercive. 
\end{lemma}

\begin{proof}
	Note that
	\begin{align*}
	\|\tilde{g}_k(u)-\tilde{g}_k(v)\|&=\Big\|\left(P^{-1}\tilde{\design}_{k:}\tilde{\design}_{k:}^TP^{-1}-(I-P)P^{-2}\mathrm{diag}(\tilde{\design}_{k:}\tilde{\design}_{k:}^T)\right)(u-v)\Big\| \\
	&\leq \Big\|\left(P^{-1}\tilde{\design}_{k:}\tilde{\design}_{k:}^TP^{-1}-(I-P)P^{-2}\mathrm{diag}(\tilde{\design}_{k:}\tilde{\design}_{k:}^T)\right)\Big\|\|u-v\| \\
	&\leq \Big\|\frac{1}{p_m^2}\left(\tilde{\design}_{k:}\tilde{\design}_{k:}^T-(1-p_m)\mathrm{diag}(\tilde{\design}_{k:}\tilde{\design}_{k:}^T)\right)\Big\|\|u-v\| \\
	&\leq \frac{1}{p_m^2}\|\tilde{\design}_{k:}\|^2\|u-v\|,
	\end{align*}
	where we have used the Weyl inequality in the last step. 
	
	One can thus choose $L_{k,D}=\frac{1}{p_m^2}\|\tilde{\design}_{k:}\|^2$ and $$L=\sup_{k,D} L_{k,D}\leq \frac{1}{p_m^2}\sup_k\|\design_{k:}\|^2 \leq \frac{1}{p_m^2} \max_k \|\design_{k:}\|^2$$
	
	Then, let us prove that the primitive of the adjusted gradient $\tilde{g}_k$ is convex. To do this, we check that the derivative of $\tilde{g}_k$ is definite positive:
	$$\frac{\partial}{\partial\coeff}\tilde{g}_k(\coeff)=\frac{1}{p^2}\left(\missdesignzero_{k:}\missdesignzero_{k:}^{T}-{(1-p)}\textrm{diag}\left(\missdesignzero_{k:}\missdesignzero_{k:}^{T}\right)\right)$$
	since $\left(\missdesignzero_{k:}\missdesignzero_{k:}^T-(1-p)\textrm{diag}\left(\missdesignzero_{k:}\missdesignzero_{k:}^{T}\right)\right)$ is positive semi-definite. Indeed, 
	\begin{align*}
	&v^T\left(\missdesignzero_{k:}\missdesignzero_{k:}^T-(1-p)\textrm{diag}\left(\missdesignzero_{k:}\missdesignzero_{k:}^{T}\right)\right)v\geq 0 \\
	&\Leftrightarrow \sum_{j=1}^\dimvar p\tilde{\design_{kj}}^2v_j^2+2\sum_{1\leq j<l\leq \dimvar}\missdesignzero_{kj}\missdesignzero_{kl}v_jv_l \geq 0 \\
	&\Leftrightarrow\left(\sum_{j=1}^\dimvar\sqrt{p}\missdesignzero_{kj}v_j\right)^2\geq 0,
	\end{align*}
	using $p^2\left(\missdesignzero_{kj}\right)^2\left(\missdesignzero_{kj}\right)^2\leq \left(\missdesignzero_{kj}\right)^2\left(\missdesignzero_{kl}\right)^2$ since $p\leq1$.
\end{proof}

\section{Add-on to \Cref{sec:exp}: Lipschitz constant computation}
\label{appsec:Additional-figures}

The Lipschitz constant $L$ given in \eqref{eq:smooth_const} is either computed from the complete covariates (oracle estimate)
$\hat{L}_n^\textrm{OR}=\frac{1}{p_m^2}\max_{1\leq k\leq n}\|X_{k:}\|^2,$
or estimated from the incomplete data matrix,
$
\hat{L}_n^\textrm{NA}=\frac{1}{\hat{p}_m^2}\max_{1\leq k\leq n}\frac{\|\tilde{X}_{k:}\|^2d}{\sum_j D_{kj}}, 
$
with $\hat{p}_m=\min_{1\leq j \leq d} \hat{p}_j$, and  $\hat{p}_j=\frac{\sum_{k}D_{kj}}{n}$.
In $\hat{L}_n^\textrm{NA}$, the squared norm of each row $\|\tilde{X}_{k:}\|^2$ is divided by the  proportion of observed values $\frac{d}{\sum_j D_{kj}}$. This way, the value of $\|\tilde{X}_{k:}\|^2$ is renormalized, by taking into account that some rows may contain more missing values than others. Note that theoretically the step size has to satisfy $\alpha\leq \frac{1}{2\hat{L}_n^\textrm{NA}}$, thus $\hat{L}_n^\textrm{NA}$ may be overestimated but should not be underestimated at the risk of instability in Algorithm \ref{alg:SGDav}. 
\Cref{fig:synthdatap07} shows that using  a slightly overestimated Lipschitz constant estimate does not deteriorate the convergence obtained using the oracle estimate. 

\begin{figure}[H]
	\begin{center}
		\includegraphics[width=0.95\textwidth]{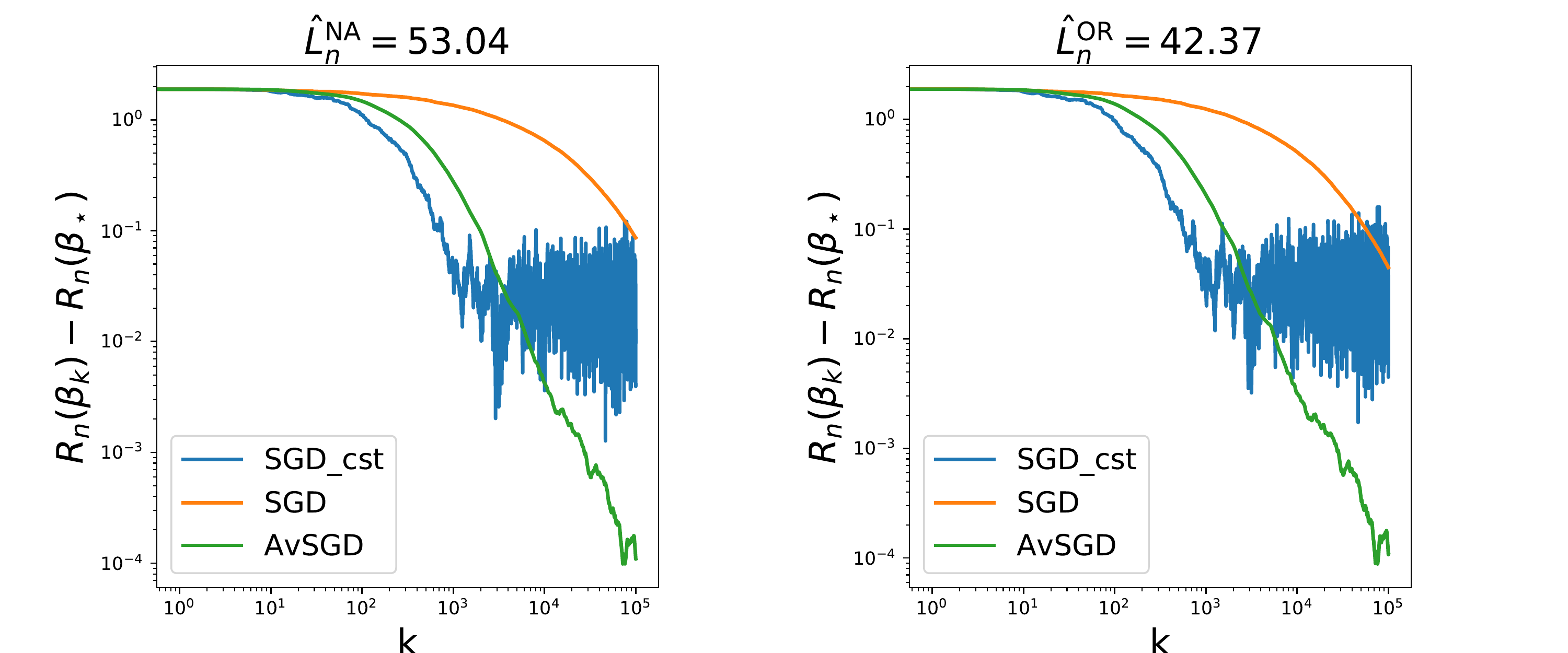}
		\caption{ \label{fig:synthdatap07}  Empirical excess risk $\left(R_\dimob(\coeff_k)-R_\dimob(\coeff^{\star})\right)$ given $\dimob$ for synthetic data ($n=10^5$, $d=10$) when there is 30\% MCAR data, with 1 pass over the data and estimating the Lipschitz  constant.}
	\end{center}
\end{figure}

\section{Add-on to \Cref{sec:exp}: Handling polynomial missing features}
\label{appsec:poly_features}

The debiased averaged SGD algorithm proposed in \Cref{sec:algo} can be further extended to the case of polynomial features by using a different debiasing than in \Cref{eq:gradMCARgeneral}. 

For example, in dimension $d=2$, with second-order polynomial features, the interaction effect of $X_{k1}X_{k2}$ and the effects of $X_{k1}^2$, $X_{k2}^2$ are accounted, so the augmented matrix design can be written as 
	$$( X_{:1} | X_{:2} | X_{:1}X_{:2} | X_{:1}^2 | X_{:2}^2 )^T.$$
	 Then, the ``descent'' direction at iteration $k$ in  \Cref{eq:gradMCARgeneral} should be chosen as
	 $$U^{\odot -1}\odot \tilde{X}_{k:}\tilde{X}_{k:}^T\beta_k-\mathrm{diag}(U)^{\odot -1}\odot\tilde{X}_{k:}y_{k}$$
	 where
	\vspace{-1em}\[U=\begin{pmatrix}
	p_1 & p_1p_2 & p_1p_2 & p_1 & p_1p_2 \\
	p_1p_2  & p_2 & p_1p_2  & p_1p_2 & p_2 \\
	p_1p_2  & p_1p_2 & p_1p_2 & p_1p_2 & p_1p_2 \\
	p_1 & p_1p_2  &  p_1p_2  & p_1 & p_1p_2 \\
	p_1p_2  & p_2 & p_1p_2  & p_1p_2 & p_2
	\end{pmatrix},\]
and $\mathrm{diag}(U)$ denotes the vector formed by the diagonal coefficients of $U$ and $U^{\odot -1}$ stands for the matrix formed of the inverse coefficients of $U$.

\paragraph{Synthetic data}  Considering a second-order model, we simulate data according to $y=(X_{:1}X_{:2} | X_{:1}^2 | X_{:2}^2 )^T\beta^\star + \epsilon$. An additional experiment is given in~\Cref{fig:mixedeffects} in \Cref{appsec:Additional-figures}, illustrating that Algorithm~\ref{alg:SGDav} still achieves a rate of $\mathcal{O}\left(\frac{1}{n}\right)$ while dealing with polynomial features of degree 2.

\begin{figure}[H]
	\begin{center}
		\includegraphics[width=0.7\textwidth]{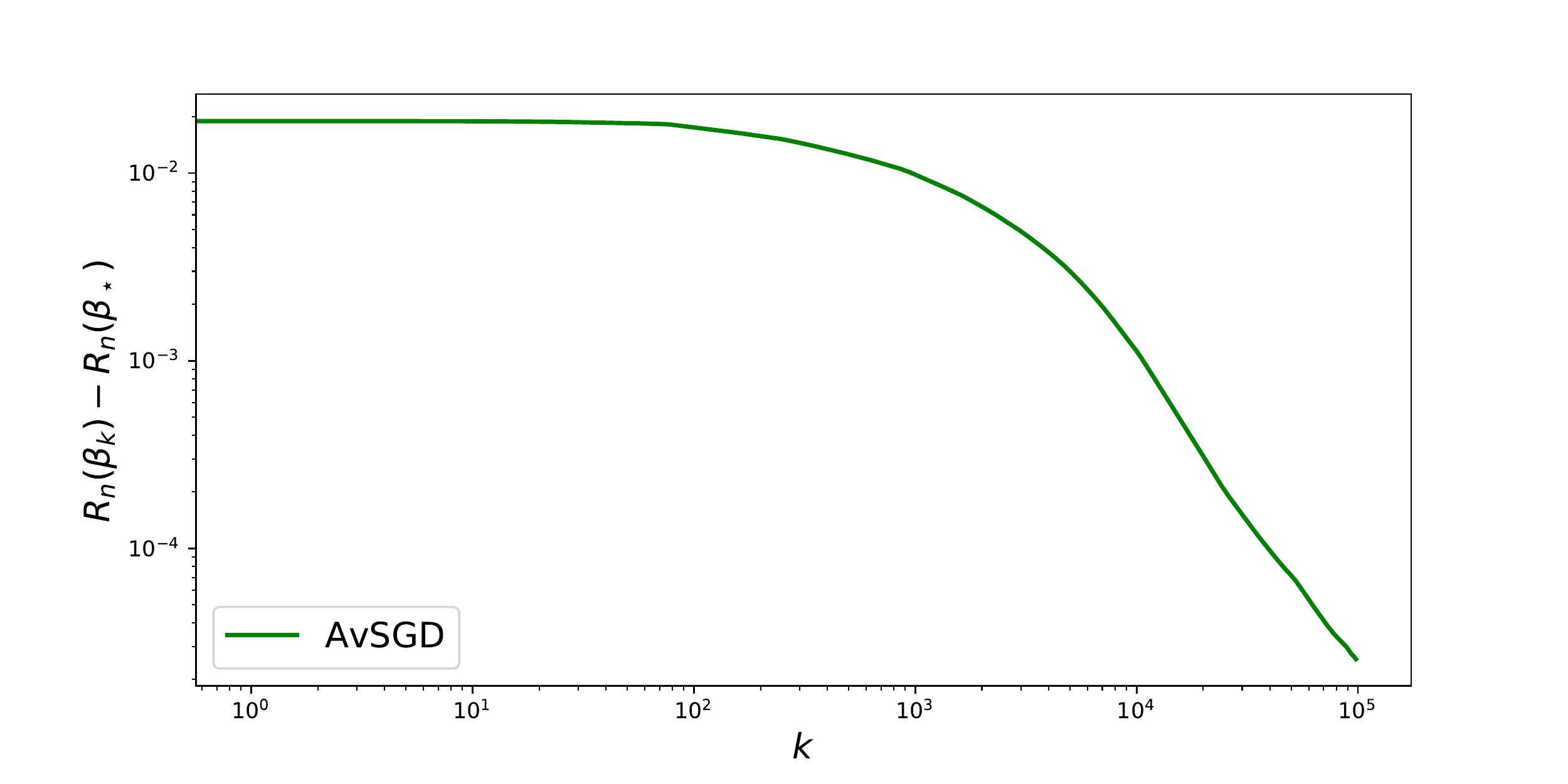}
		\caption{ \label{fig:mixedeffects}  Empirical excess risk $\left(R_\dimob(\coeff_k)-R_\dimob(\coeff^{\star})\right)$ given $\dimob$ for synthetic data ($n=10^5$, $d=10$) when the model accounts mixed effects.}
	\end{center}
\end{figure}

\paragraph{Real dataset} About large-scale setting there is no computational barrier to apply the proposed method in high dimension, as the computational cost is similar to standard SGD strategies without missing data. These are computationally cheap at each iteration and particularly relevant on large datasets. 
In this section, we propose to run the proposed algorithm on the superconductivity dataset as in \Cref{subsubsec:superconductivity}. $30\%$ of missing values are uniformly introduced in the initial 81 features, with $n= 21263$. However, here we consider polynomial features of order 2, which increases the initial dimension 81 to 3400.

The empirical proportions of missing values for each variable in the resulting dataset are represented on \Cref{fig:superconductivity_poly_proportion}, and the observed convergence rate for one pass on the data is displayed in \Cref{fig:superconductivity_poly}. With the same numerical complexity,
\Cref{alg:SGDav} performs as well as an averaged SGD strategy run on the complete observations, whereas a standard SGD strategy run on imputed-by-0 data saturates far from the optimum.

\begin{figure}[H]
	\begin{center}
		\includegraphics[width=0.7\textwidth]{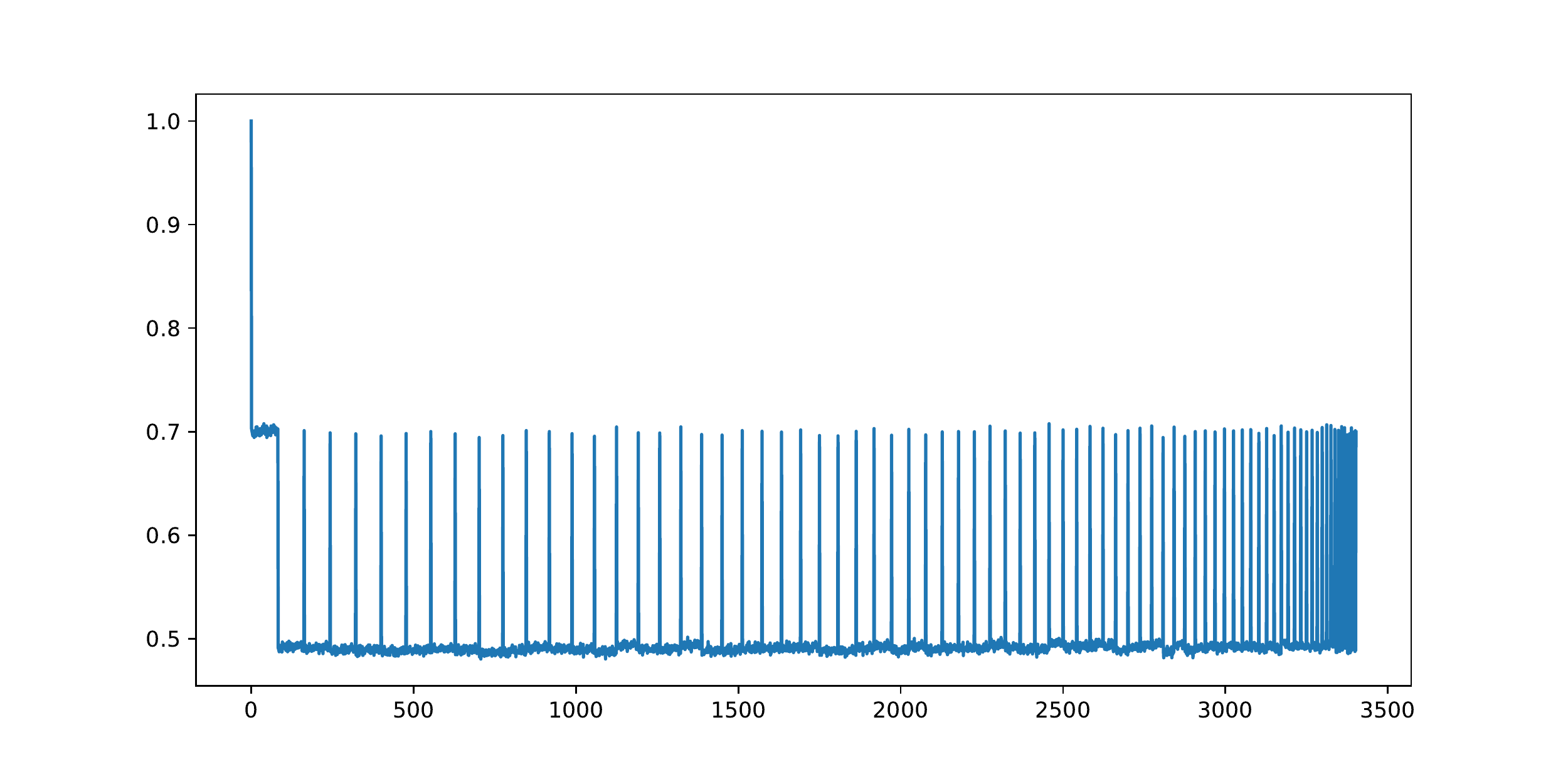}
		\caption{ \label{fig:superconductivity_poly_proportion} Proportion of missing values for the polynomial features of degree 2 on the superconductivity dataset, when the initial missingness proportion on the raw features is $30\%$.}
	
	\end{center}
\end{figure}

\begin{figure}[H]
	\begin{center}
		\includegraphics[width=0.7\textwidth]{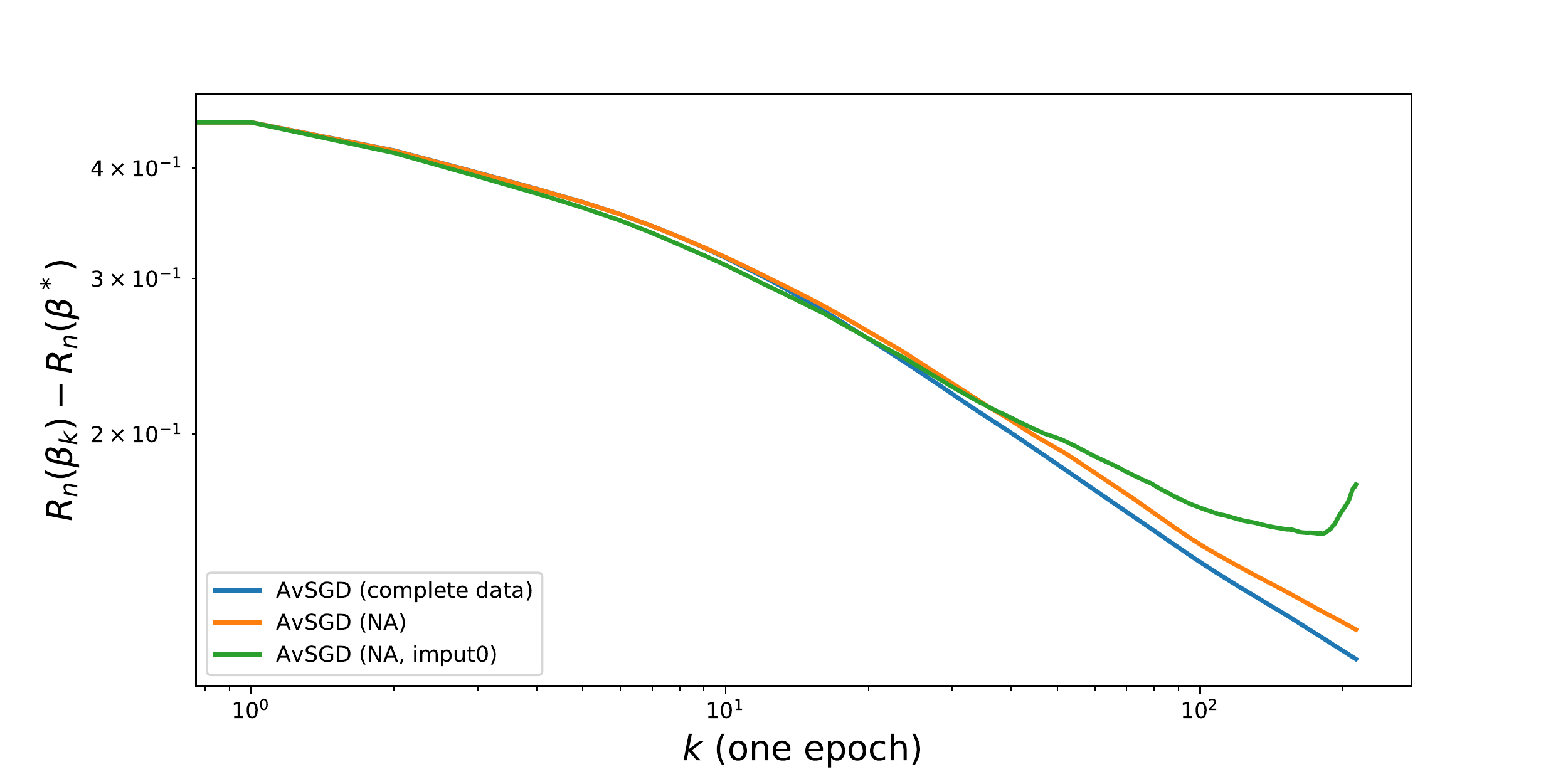}
		\caption{ \label{fig:superconductivity_poly}  Empirical excess risk $\left(R_\dimob(\coeff_k)-R_\dimob(\coeff^{\star})\right)$ given $\dimob$ for the superconductivity dataset ($n=21263$) (containing 81 initial features) and $d=3403$ with polynomial features of degree 2. Three different algorithms are compared: an averaged SGD on complete data (blue), the proposed debiased averaged SGD \Cref{alg:SGDav} (orange) and an averaged SGD run on imputed-by-0 data without any debiasing (green).}
	\end{center}
\end{figure}

\section{Add-on to \Cref{sec:exp}: Description of the TraumaBase data variables}
\label{appsec:Traumabase}
The variables of the TraumaBase dataset which are used in experiments are the following:
\begin{itemize}[label=\textbullet]
	\item \textit{Lactate}: The conjugate base of lactic acid. 
	\item \textit{Delta.Hemo}: The difference between the homoglobin on arrival at hospital and that in the ambulance. 
	\item \textit{VE}: A volume expander is a type of intravenous  therapy that has the function of providing volume for the circulatory system. 
	\item \textit{RBC}: A binary index which indicates whether the transfusion of Red Blood Cells Concentrates is performed. 
	\item \textit{SI}: Shock index indicates level of occult shock based on heart rate (\textit{HR}) and systolic blood pressure (\textit{SBP}). $SI=\frac{HR}{SBP}$. Evaluated on arrival at hospital.
	\item \textit{HR}: Heart rate measured on arrival of hospital. 
	\item \textit{Age}: Age.
	
\end{itemize}

\end{document}

%% file: header.tex
\usepackage{amsmath, amssymb}
\usepackage{graphicx} 
\graphicspath{{img/}} 
\usepackage{enumerate}
\usepackage{enumitem}
\usepackage{amsthm}
\usepackage{stmaryrd}

\newcommand{\design}{X}
\newcommand{\missdesignzero}{\tilde{X}}
\newcommand{\coeff}{\beta}
\newcommand{\pred}{y}
\newcommand{\noise}{\epsilon}
\newcommand{\dimob}{n}
\newcommand{\dimvar}{d}

\newcommand{\pas}{\alpha}

\newcommand{\gradtR}{\tilde{g}_{k}}
\newcommand{\gradRbis}{g_k}

\DeclareMathOperator*{\argmin}{arg\,min}

\usepackage{cleveref}
\crefname{lemma}{Lemma}{Lemmas}
\crefname{fact}{Fact}{Facts}
\crefname{theorem}{Theorem}{Theorems}
\crefname{corollary}{Corollary}{Corollaries}
\crefname{claim}{Claim}{Claims}
\crefname{example}{Example}{Examples}
\crefname{problem}{Problem}{Problems}
\crefname{definition}{Definition}{Definitions}
\crefname{assumption}{Assumption}{Assumptions}
\crefname{subsection}{Subsection}{Subsections}
\crefname{section}{Section}{Sections}
\crefname{algorithm}{Algorithm}{Algorithms}
\crefname{algocf}{alg.}{algs.}
\Crefname{algocf}{Algorithm}{Algorithms}
\crefname{proposition}{Proposition}{Propositions}

\newtheorem{theorem}{Theorem}
\newtheorem{lemma}[theorem]{Lemma}

\newtheoremstyle{remarks}
{\topsep}   
{\topsep}   
{}  
{0pt}       
{\bfseries} 
{}         
{5pt plus 1pt minus 1pt} 
{}          
\theoremstyle{remarks}
\newcounter{remarks}
\newtheorem{remark}[remarks]{Remark}

\usepackage[normalem]{ulem}
\usepackage[dvpisnames,svgnames]{xcolor}

\usepackage[colorinlistoftodos,textwidth=2.3cm]{todonotes}

\usepackage{algorithm}
\usepackage{algorithmic}
\usepackage{multicol}
\usepackage{tasks}

%% file: main.tex
\begin{abstract}
Stochastic gradient algorithm is a key ingredient of many machine learning methods, particularly appropriate for large-scale learning.
However, a major caveat of large data is their incompleteness.
We propose an averaged stochastic gradient algorithm handling missing values in linear models. This approach has the merit to be free from the need of any data distribution modeling and to account for heterogeneous missing proportion.
In both streaming and finite-sample settings, we prove that this algorithm achieves convergence rate of $\mathcal{O}(\frac{1}{n})$ at the iteration $n$, the same as without missing values. 
We show the convergence behavior and the relevance of the algorithm not only on synthetic data but also on real data sets, including those collected from medical register. 
 \end{abstract}

\section{Introduction} \label{sec:intro}
\input{intro}

\section{Problem setting} \label{sec:setting}
\input{setting}

\section{Averaged SGD with missing values} \label{sec:algo}

\input{algo}

\section{Theoretical results} \label{sec:results}
\input{theoretical_results}

\section{Experiments} \label{sec:exp}
\input{xp}

\bibliographystyle{plainnat}
\bibliography{biblio}

%% file: intro.tex
Stochastic gradient algorithms (SGD) \citep{robbins1951stochastic}
play a central role in machine learning problems, due to their cheap computational cost and memory per iteration. 
There is a vast literature on its variants, for example using averaging of the iterates \citep{polyak1992acceleration}, some robust versions of SGD \citep{nemirovski2009robust, juditsky2011solving} or adaptive gradient algorithms like Adagrad \citep{duchi2011adaptive}; and on theoretical  guarantees of those methods \citep{moulines2011non,bach2013non,dieuleveut2017harder,shamir2013stochastic,hazan2011beyond,needell2014stochastic}. 
More globally, averaging strategies have been used to stabilize the algorithm behaviour and reduce the impact of the noise, giving better convergence rates without requiring strong convexity. 

The problem of missing values is ubiquitous in large scale data analysis. 
One of the key challenges in the presence of missing data is to deal with the half-discrete nature of the data which can be seen as a mixed of continuous data (observed values) and categorical data (the missing values). In particular for gradient-based methods, the risk minimization with incomplete data becomes intractable and the usual results cannot be directly applied. 

\paragraph{Context.} 

In this paper, we consider a linear regression model, for $i\geq 1$, 
\begin{equation}\label{eq:linearmodel}
\pred_i=\design_{i:}^T\coeff^\star+\noise_i, 
\end{equation}
parametrized by $\coeff^\star \in \mathbb{R}^\dimvar$, where $\pred_i\in \mathbb{R}$,  $\noise_i\in \mathbb{R}$ is a real-valued centered noise and $\design_{i:}\in \mathbb{R}^\dimvar$ stands for the real covariates of the $i$-th observation. 
The $(\design_{i:})$'s are assumed to be only partially known, since some covariates may be missing: our objective is to derive stochastic algorithms for estimating the parameters of the linear model, which handle missing data, and come with strong theoretical guarantees on excess risk.

\paragraph{Related works.} 
There is a rich literature on  handling missing values \citep{little2019statistical} and yet there are still some challenges even for linear regression models. This is all the more true as we consider such models for large sample size or in high dimension. There are very few regularized versions of regression that can deal with missing values.  
A classical approach to estimating parameters with missing values consists in maximizing the observed likelihood, using  for instance an Expectation Maximization algorithm \citep{dempster1977maximum}. Even if this approach can be implemented to scale for large datasets see for instance \citep{cappe2009line},
one of its main drawbacks is to rely on strong parametric assumptions for the covariates distributions.
Another popular strategy to fix the missing values issue consists in predicting the missing values to get a completed data and then in applying the desired method. 
However matrix completion is a different problem from estimating parameters and can lead to uncontrolled bias and undervalued variance of the estimate  \citep{little2019statistical}. In the regression framework, \citet{jones1996indicator} studied the bias induced by naive imputation.
 
In the settings of the Dantzig selector \citep{rosenbaum2010sparse} and {LASSO} \citep{loh2011high}, 
another solution consists in naively imputing by 0 the incomplete matrix and modifying the algorithm used in the complete case to account for the imputation error. 
Such a strategy has also been studied by \citet{ma2017stochastic} for SGD  in the context of linear regression with missing values and with finite samples: the authors used debiased gradients, in the same spirit as the covariance matrix debiasing considered by \citet{loh2011high} in a context of sparse linear regression, or by \citet{koltchinskii2011nuclear} for matrix completion.
This modified version of the SGD algorithm \citep{ma2017stochastic} is {conjectured} to converge in expectation to the ordinary least squares estimator, achieving the rate of $\mathcal{O}(\frac{\log n}{\mu n})$ at iteration $n$ for the excess empirical risk, assumed to be $\mu$-strongly convex in that work. 
However, their algorithm requires a step choice relying on the knowledge of the strong-convexity constant $\mu$ which is often intractable for large-scale settings.

\vspace{-0.5em}
\paragraph{Contributions.} 
\begin{itemize}[topsep=0pt,itemsep=1pt,leftmargin=*]
	\item We develop a debiased averaged SGD to perform (regularized) linear regression either streaming or with finite samples, when covariates are missing.  The approach consists in imputing the covariates with a simple imputation and using debiased gradients accordingly.
	\item Furthermore, the design is allowed to be contaminated 
	by heterogeneous missing values:  each covariate may have a different probability to be missing. This encompasses the classical homogeneous Missing Completely At Random (MCAR) case, where the missingness is independent of any covariate value.
	\item This algorithm comes with theoretical guarantees: we establish convergence in terms of generalization risk  at the rate $1/n$ at iteration $n$.  This rate is remarkable as it is (i) optimal w.r.t.~$n$, (ii) free from any bad condition number (no strong convexity constant is required), and (iii) similar to the rate of averaged SGD without any missing value. 
	\item In terms of performance with respect to the missing entries proportion in large dimension, our strategy results in an error provably several orders of magnitude smaller than the best possible algorithm that would only rely on complete observations.
	\item   We show the relevance of the proposed approach and its convergence behavior on numerical  applications and its efficiency on real data; including the TraumaBase$^{\mbox{\normalsize{\textregistered}}}$ dataset to assist doctors in making real-time decisions in the management of severely traumatized patients. The code to reproduce all the simulations and numerical experiments is available on \url{https://github.com/AudeSportisse/SGD-NA}.
	
\end{itemize}

%% file: setting.tex
In this paper, we consider either the streaming setting, i.e.\ when the data comes in as it goes along, or the finite-sample setting, i.e.\ when the data size is fixed and form a finite design matrix $\design = (X_{1:} | \hdots | X_{\dimob :})^T \in \mathbb{R}^{ \dimob \times \dimvar}$ ($\dimob>\dimvar$). We define $\mathcal D_n := \sigma\left( ( \design_{i:} , y_i ) , i = 1, \dots, n \right) $ the $ \sigma- $field generated by $n$ observations. 
We also denote   $\preccurlyeq$  the partial order between self-adjoint operators, such that $A\preccurlyeq B$ if $B-A$ is positive semi-definite.

Given observations as in \eqref{eq:linearmodel} and defining $f_i(\beta) := \left( \langle \design_{i:} , \coeff \rangle - y_i \right)^2/2$, the (unknown) linear model parameter satisfies:
\begin{equation}
\label{pb:online2}
\coeff^\star 
= \argmin_{\coeff \in \mathbb{R}^\dimvar} \left\{R(\beta) :=  \mathbb{E}_{(\design_{i:},y_i)}\left[f_i(\coeff)\right] \right\},
\end{equation}
 where $\mathbb{E}_{(X_{i:},y_i)}$ denotes the expectation over the distribution of $(X_{i:},y_i)$ (which is independent of $ i $ as the observations are assumed to be i.i.d.).

In this work, the covariates are assumed to contain missing values, so one in fact observes ${X}^{\textrm{NA}}_{i:} \in (\mathbb{R}\cup\{\mathtt{NA}\})^\dimvar$ instead of $X_{i:}$, as
$\design^{\textrm{NA}}_{i:} := \design_{i:}  \odot D_{i:} + \mathtt{NA}  (\mathbf{1}_{\dimvar} - D_{i:}),
$
where $\odot$ denotes the element-wise product, $\mathbf{1}_{\dimvar} \in \mathbb{R}^{\dimvar}$ is the vector filled with ones and $D_{i:} \in \{0,1\}^{\dimvar}$ is a binary vector mask coding for the presence of missing entries in $\design_{i:}$, i.e.\  $D_{ij} =0$ if the $(i,j)$-entry is missing in $\design_{i:}$, and $D_{ij} =1$ otherwise. We adopt the convention $\mathtt{NA} \times 0 = 0$ and $\mathtt{NA} \times 1 = \mathtt{NA}$. We consider a \emph{heterogeneous} MCAR  setting, i.e.
$D$ is modeled with a Bernoulli distribution
\begin{equation}
\label{eq:binarymaskprob}
D=\left(\delta_{ij}\right)_{1 \leq i \leq \dimob , 1\leq j \leq \dimvar} \quad \text{with} \quad \: \delta_{ij} \sim \mathcal{B}(p_j),
\end{equation}
with $1-p_j$ the probability that the $j$-th covariate is missing.

The considered approach consists in imputing the incomplete covariates by zero in $\design^{\textrm{NA}}_{i:}$, 
as  
$\missdesignzero_{i:}=\design^{\textrm{NA}}_{i:}\odot D_{i:}=X_{i:}\odot D_{i:},$
and in accounting for the imputation error in the subsequent algorithm.

%% file: algo.tex
\begin{multicols}{2}
The proposed method is detailed in Algorithm \ref{alg:SGDav}.
The impact of the naive imputation by 0 directly translates into a bias in the gradient. Consequently, at each iteration we use a debiased estimate $ \tilde{g}_{k}$. In order to stabilize the stochastic algorithm, we consider the Polyak-Ruppert~\cite{polyak1992acceleration} averaged iterates  $\bar{\coeff}_{k}=\frac{1}{k+1}\sum_{i=0}^{k}\coeff_i.$

\begin{lemma} \label{lem:unbiased}Let $(\mathcal{F}_k)_{k\geq 0}$ be the following $\sigma$-algebra, 
	 	$
\mathcal{F}_k=  \sigma (X_{1:},y_1,D_{1:}\dots,X_{k:},y_k,D_{k:}).
	 $
\label{ass:iterationtheo} 
The modified gradient $\tilde{g}_{k}(\beta_{k-1})$ in \Cref{eq:gradMCARgeneral} is $\mathcal{F}_{k}$-measurable and a.s., 
	\begin{equation*}
		\mathbb{E}\left[\tilde{g}_{k}(\coeff_{k-1})\left.\right|\mathcal{F}_{k-1}\right]= \nabla R(\coeff_{k-1}).
	\end{equation*}

\end{lemma}

\begin{algorithm}[H]
	\caption{Averaged SGD for Heterogeneous Missing Data}
	\label{alg:SGDav}
	\begin{algorithmic}
		\STATE {\bfseries Input:} data $\tilde{X}, y, \alpha$ (step size)
		\STATE Initialize $\beta_0=0_d$.
		\STATE Set $P=\mathrm{diag}\left ((p_j)_{j \in \{1,\dots,d\}}\right) \in \mathbb{R}^{d\times d}$.
		\FOR{$k=1$ {\bfseries to} $n$}		    
		    \STATE 
		    \vspace{-0.6cm}
		    \begin{multline}	    
		    \label{eq:gradMCARgeneral}		  \tilde{g}_{k}(\coeff_k)=P^{-1}\missdesignzero_{{k}:}\left(\missdesignzero_{{k}:}^TP^{-1}\coeff_k-y_{k}\right)\\ 
		    -(\mathrm{I}-P)P^{-2}\mathrm{diag}\left(\missdesignzero_{{k}:}\missdesignzero_{{k}:}^T\right)\coeff_{k}
		    \end{multline}
		    \STATE $\coeff_{k}=\coeff_{k-1}-\pas\tilde{\gradRbis}(\coeff_{k-1})$
		    \STATE $\bar{\coeff}_{k}=\frac{1}{k+1}\sum_{i=0}^{k}\coeff_i = \frac{k}{k+1} \bar{\coeff}_{k-1} + \frac{1}{k+1} \coeff_k $
		\ENDFOR
	\end{algorithmic}
\end{algorithm}
\end{multicols}
\vspace{-0.4cm}
\Cref{ass:iterationtheo} is proved in \Cref{supsec:lem:unbiased}.
Note that in the case of homogeneous MCAR data, i.e.
$p_1 = \hdots = p_d = p\in (0,1)$, the chosen direction at iteration $k$ in \Cref{eq:gradMCARgeneral} boils down to 
$
\frac{1}{p}\missdesignzero_{{k}:}\left( \frac{1}{p}\missdesignzero_{{k}:}^T\coeff_k-y_{k}\right) -\frac{1-p}{p^2}\mathrm{diag}\left(\missdesignzero_{{k}:}\missdesignzero_{{k}:}^T\right)\coeff_{k}.
$
This meets the classical debiasing terms of covariance matrices{ \cite{loh2011high, ma2017stochastic, koltchinskii2011nuclear} }.  
Note also that in the presence of complete observations, meaning that $p=1$, Algorithm \ref{alg:SGDav} matches the standard least squares stochastic algorithm.

\begin{remark}[Ridge regularization]
	\label{rem:ridge}
	Instead of minimizing the theoretical risk as in \eqref{pb:online2}, we can consider a Ridge regularized formulation:
	$\min_{\coeff \in \mathbb{R}^\dimvar} \: R(\coeff) + \lambda\|\coeff\|^2,$
	with $\lambda>0$. 
    \Cref{alg:SGDav} is trivially extended to this  framework: the debiasing term is not modified since the penalization term does not involve the incomplete data $\tilde{X}_{i:}$. This is useful in practice as no implementation is availaible for incomplete ridge regression.
\end{remark}

%% file: theoretical_results.tex
In this section, we prove convergence guarantees for  \Cref{alg:SGDav} in terms of theoretical excess risk, in both the streaming and the finite-sample settings. 
For the rest of this section, assume the following.
\begin{itemize}[topsep=0pt,itemsep=1pt,leftmargin=*]
	\item The observations $(X_{k:},y_k) \in \mathbb{R}^{d}\times \mathbb{R}$ are independent and identically distributed. 
	\item  $\mathbb{E}[\|X_{k:}\|^2]$ and $\mathbb{E}[\|y_k\|^2]$ are finite. 
	\item Let $H$ be an invertible matrix, defined by
$H:=\mathbb{E}_{(X_{k:},y_k)}[X_{k:}X_{k:}^T]. $
\end{itemize}

The main technical challenge to overcome is proving that the noise in play due to missing values is \emph{strutured} and still allows to derive convergence results for a debiased version of averaged SGD. This work builds upon the analysis made by \citet{bach2013non} for standard SGD strategies.

\subsection{Technical results}

\citet{bach2013non} proved that for least-squares regression,  {averaged SGD} converges at rate $n^{-1}$ after $ n $ iterations. In order to derive similar results,  we prove in addition to  \Cref{lem:unbiased}, \Cref{lem:noisetheo,lem:covtheo}: 
\begin{itemize}[topsep=0pt,itemsep=1pt,leftmargin=*,noitemsep]
	\item \Cref{lem:noisetheo} shows that the noise induced by the imputation by zeros and the subsequent transformation results in a \emph{structured noise}.  This is the most challenging part technically: having a structured noise is fundamental to obtain convergence rates scaling as $ n^{-1} $ -- in the unstructured case the convergence speed is only $ n^{-1/2} $ \citep{dieuleveut2017harder}. 
	\item \Cref{lem:covtheo} shows that the adjusted random gradients $ \gradtR(\coeff) $ are almost surely  \emph{co-coercive} ~\citep{Zhu_Mar_1995} i.e., for any $k$, there exists a random ``primitive'' function
	$ \tilde{f}_{k}$ which is a.s.~convex and smooth, and such that $ \tilde{g}_k  =\nabla \tilde{f}_{k} $ .  Proving that $ \tilde{f}_{k}$ is a.s.~convex is an important step which was missing in the analysis of \citet{ma2017stochastic}.
\end{itemize}

\begin{lemma}
\label{lem:noisetheo} 
 The additive  noise process $(\tilde{g}_{k}(\coeff^\star))_k$ with $\beta^\star$ defined in \eqref{pb:online2} is $\mathcal{F}_{k}-$measurable and,
	\begin{enumerate}
		\item \label{ass:noisezeroexp} $\forall k\geq 0, \: \mathbb{E}[\tilde{g}_{k}(\coeff^\star) \left.\right| \mathcal{F}_{k-1}]=0$ a.s..
		\item \label{ass:noisebound} $ \forall k\geq 0, \: \mathbb{E}[\|\tilde{g}_{k}(\coeff^\star)\|^2 \left.\right| \mathcal{F}_{k-1}]$ is a.s. finite.
		\item \label{ass:noisecov} $\forall k\geq 0, \: \mathbb{E}[\tilde{g}_{k}(\coeff^\star){\tilde{g}_{k}(\coeff^\star)^T}]\preccurlyeq C(\beta^\star)=c(\beta^\star)H$.
	\end{enumerate}
\end{lemma}

\begin{proof}[Sketch of proof (\Cref{lem:noisetheo})]
	 Property \ref{ass:noisezeroexp} easily followed from \Cref{lem:unbiased} and the definition of $\beta^\star$.
	 Property \ref{ass:noisebound} can be obtained with similar computations as in \citep[Lemma 4]{ma2017stochastic}. Property \ref{ass:noisecov} cannot  be directly  derived from Property~\ref{ass:noisebound}, since $\tilde{g}_{k}(\coeff^\star){\tilde{g}_{k}(\coeff^\star)^T}\preccurlyeq \|\tilde{g}_{k}(\coeff^\star)\|^2I$ leads to an insufficient upper bound. Proof relies on decomposing the external product $\tilde{g}_{k}(\coeff^\star){\tilde{g}_{k}(\coeff^\star)^T}$ in several terms and obtaining the control of each, involving technical computations. 
\end{proof}

\begin{lemma}
	 \label{lem:covtheo}  For all $k \geq 0$, given the binary mask $D$, the adjusted gradient $\tilde{g}_k(\coeff)$ is a.s. $L_{k,D}$-Lipschitz continuous, i.e.\ for all $u, v \in \mathbb {R}^d$,
	$\|\tilde{g}_{k}(u)-\tilde{g}_{k}(v)\|\leq L_{k,D}\|u-v\|~\textrm{a.s.}.
	$
	Set
	\begin{equation}
	\label{eq:smooth_const}
	L:=\sup_{k,D} L_{k,D} \leq \frac{1}{p_m^2} \max_k \|\design_{k:}\|^2~\textrm{a.s.}.
	\end{equation}
	In addition, for all $k \geq 0$, $\tilde{g}_{k}(\beta)$ is almost surely co-coercive. 
\end{lemma}

\Cref{lem:noisetheo,lem:covtheo} are respectively proved in  \Cref{appsec:lem:noisetheo,appsec:lem:covtheo-app}, and can be combined with Theorem 1 in \cite{bach2013non} in order to prove the following theoretical guarantees for  \Cref{alg:SGDav}.

\subsection{Convergence results}

The following theorem quantifies the convergence rate of \Cref{alg:SGDav} in terms of excess risk.

\begin{theorem}[Streaming setting]\label{theo:Bachresultgen}
	Assume that for any $i$, $\|X_{i:}\|\leq \gamma$ almost surely for some $\gamma>0$. For any constant step-size $\alpha \leq \frac{1}{2L}$,  \Cref{alg:SGDav} ensures that, for any $ k\geq 0 $:
	$$\mathbb{E}\left [R\left (\bar{\coeff}_k\right )-R(\coeff^\star)\right ]\leq\frac{1}{2k}\left(\frac{\sqrt{c(\beta^\star)\dimvar}}{1-\sqrt{\pas L}}+\frac{\|\coeff_0-\coeff^\star\|}{\sqrt{\pas}}\right)^2,$$
	with $L$ given in \Cref{eq:smooth_const}, $p_m=\min_{j=1, \dots d} \  p_j$ and 
	\begin{equation}
	\label{eq:constcgen}
	c(\beta^\star)=\frac{\mathrm{Var}(\epsilon_k)}{p_m^2}
	+\left(\frac{(2+5 p_m)  (1-p_m)}{p_m^3}\right)\gamma^2\|\coeff^\star\|^2 .
	\end{equation}
\end{theorem}

Note that in \Cref{theo:Bachresultgen}, the expectation is taken over the randomness of the observations $ (X_{i:},y_i,D_{i:})_{1\le i \le k} $.  
The bounded features assumption in \Cref{theo:Bachresultgen} is mostly convenient for the readability, but it can be relaxed at the price of milder but more technical assumptions and proofs (typically bounds on quadratic mean instead of a.s.\ bounds). 


\begin{remark}[Finite-sample setting] \label{theo:Bachempiricalgen}
    Similar results as \Cref{theo:Bachresultgen} can be derived in the case of finite-sample setting. For the sake of clarity, they are made explicit hereafter:
	for any constant step-size $\alpha \leq \frac{1}{2L}$,   \Cref{alg:SGDav} ensures that for any $ k\le n $:
$	\mathbb{E}\left [R(\bar{\coeff}_k)-R(\coeff^\star)] | \mathcal D_n\right ] 	\leq\frac{1}{2k}\left(\frac{\sqrt{c(\beta^\star)\dimvar}}{1-\sqrt{\pas L}}+\frac{\|\coeff_0-\coeff^\star\|}{\sqrt{\pas}}\right)^2
$	
	with $L$  given in \Cref{eq:smooth_const} and 
    $
    	c(\beta^\star)=\frac{\mathrm{Var}(\epsilon_k)}{p_m^2}
	    +\left(\frac{(2+5 p_m)  (1-p_m)}{p_m^3}\right)\max_{1\leq i \leq n} \|X_{i:}\|^2\|\coeff^\star\|^2 .
	$
\end{remark}	

\paragraph{Convergence rates for the iterates.}
	Note that if a Ridge regularization is considered, the regularized function to minimize $R(\coeff) + \lambda\|\coeff\|^2$ is $2\lambda$-strongly convex. \Cref{theo:Bachresultgen} and \Cref{theo:Bachempiricalgen} then directly provide the following bound on the iterates:
	$\mathbb{E}\left[ \left\|\overline{\beta}_k-\beta^\star\right\|^2\right] \leq \frac{1}{2\lambda k}\left(\frac{\sqrt{c(\beta^\star)\dimvar}}{1-\sqrt{\pas L}}+\frac{\|\coeff_0-\coeff^\star\|}{\sqrt{\pas}}\right)^2.$

\paragraph{Additional comments.} We highlight the following points:
\begin{itemize}[topsep=0pt,itemsep=1pt,leftmargin=*]
	\item  In \Cref{theo:Bachresultgen}, the expected excess risk is upper bounded by (a) \emph{a variance} term, that grows with the noise variance and is increased by the missing values, and (b) \emph{a bias} term, that accounts for the importance of the initial distance between the starting point $ \coeff_0 $ and the optimal one $ \coeff^{\star} $. 
	\item  The optimal convergence rate is achieved for a \emph{constant} learning rate $ \alpha $. One could for example choose $ \alpha = \frac{1}{2L} $, that does \emph{not decrease} with the number of iterations. In such a situation, both the \emph{bias}  and \emph{variance} terms scale as $ k^{-1} $. Remark that convergence of the averaged SGD with constant step-size only happens for least squares regression, because the un-averaged iterates converge to a limit distribution whose mean is exactly $ \beta^* $~\cite{bach2013non,dieuleveut2017bridging}.
	\item The expected risk scales as $ n^{-1} $ after $ n  $ iterations, without strong convexity constant involved. 
	\item For the generalization risk $ R $, this rate of $ n^{-1} $  is known to be statistically optimal for least-squares regression: under reasonable assumptions, no algorithm, even more complex than averaged SGD or without missing observations, can have a better dependence in $ n $ \cite{tsybakov2003optimal}. 
	\item In the complete case, i.e. when $p_1=\hdots=p_d =1$,  \Cref{theo:Bachresultgen,theo:Bachempiricalgen} meet the results from \citet[Theorem 1]{bach2013non}. Indeed, in such a case, $c(\beta^\star)=\mathrm{Var}(\epsilon_k)$. 
	\item The noise variance coefficient $c(\beta^{\star})$ includes (i) a first term as a classical noise one, proportional to the model variance, and increased by the missing values occurrence to $\frac{\mathrm{Var}(\epsilon_k)}{p_m^2}$; (ii) the second term is upper-bounded by $\frac{ 7 (1-p_m)}{p_m^3}\cdot\gamma^2\|\coeff^\star\|^2$ corresponds to the multiplicative noise induced by the imputation by 0 and 
	gradient debiasing.
	It naturally increases as the radius $\gamma^2$ of the observations increases (so does the imputation error), and vanishes if there are no missing values ($p_m=1$).
\end{itemize}

\begin{remark}[Only one epoch] 
\label{rem:noseveralpasses}
It is important to notice that in a finite-sample setting, as covered by Remark \ref{theo:Bachempiricalgen}, given a maximum number of $ n $ observations,
our convergence rates are only valid for $ k\le n $: the theoretical bound holds only for one pass on the input/output pairs. Indeed, afterwards, we cannot build unbiased gradients of the risk.
\end{remark}

\subsection{What about empirical risk minimization (ERM)?}
\label{rem:ERM}

\paragraph{Theoretical locks.}
{Note that the translation of the results in  \Cref{theo:Bachempiricalgen} in terms of empirical risk convergence is still an open issue. The heart of the problem is that it seems  really difficult to obtain a sequence of unbiased gradients of the empirical risk.
\begin{itemize}[topsep=0pt,itemsep=1pt,leftmargin=*]
    \item Indeed, to obtain unbiased gradients, the data should be processed \emph{only once} in \Cref{alg:SGDav}: if we consider the gradient of the loss with respect to an observation $k$, we obviously need the binary mask $D_k$ and the current point $\coeff_{k-1}$ to be independent for the correction relative to the missing entries to make sense. As a consequence, no sample can be used twice -
    in fact, running multiple passes over a finite sample could result in over-fitting the missing entries.
    \item Therefore, with a finite sample at hand, the sample used at each iteration should be chosen \emph{without replacement} as the algorithm runs. But even in the complete data case, sampling without replacement induces a bias on the chosen direction~\cite{gurbuzbalaban2015random,Jain19}. Consequently, \Cref{lem:unbiased} does not hold for the empirical risk instead of the theoretical one. This issue is not addressed in \cite{ma2017stochastic}, {unfortunately making the proof of their result invalid/wrong.} 
\end{itemize}}

\paragraph{Comparison to \citet{ma2017stochastic}.}
Leaving aside the last observation, we can still comment on the bounds in  \cite{ma2017stochastic} for the empirical risk without averaging.
As they do not use averaging but only the last iterate, their convergence rate (see Lemma 1 in their paper) is only studied for $ \mu- $strongly convex problems and is expected to be larger (i) by a factor $ \mu^{-1} $, due to the choice of their decaying learning rate,  and (ii) by a $ \log n $ factor due to using the last iterate and not the averaged one \citep{shamir2013stochastic}.
Moreover, the strategy of the present paper does not require to access the strong convexity constant, which is generally out of reach, if no explicit regularization is used. 
More marginally, we provide the proof of the co-coercivity of the adjusted gradients (\Cref{lem:covtheo}), which is required to derive the convergence results, and which was also missing in  \citet{ma2017stochastic}. {A more detailed discussion on the differences between the two papers is given in \Cref{app:MN}.}

\paragraph{ERM hindered by NA.}
{It is also interesting to point out that with missing features,
\emph{neither} the generalization risk $ R $, \emph{nor} the empirical risk $ R_n $ are observed (i.e., only approximations of their values or gradients can be computed). As a consequence,   one cannot expect to minimize those functions with unlimited accuracy. This stands in contrast to the \emph{complete observations setting}, in which the empirical risk $ R_n $ is known exactly.
As a consequence, with missing data, empirical risk loses its main asset - being an observable function that one can minimize with high precision. Overall it is both more natural and easier to focus on the generalization risk.}

\subsection{On the impact of missing values}

\paragraph{Marginal values of incomplete data.}
	An important question in practice is to understand how much information has been lost because of the incompleteness of the observations. In other words, it is better to access 200 input/output pairs with a probability 50\% of observing each feature on the inputs, or to observe 100 input/output pairs with complete observations? 
	
	Without missing observations, the variance bound in the expected excess risk is given by \cref{theo:Bachresultgen} with $ p_m=1 $: it scales as $   O\left(\frac{\mathrm{Var}(\epsilon_k) d}{k} \right ),$
	while with missing observations it increases to $ O\left (\frac{\mathrm{Var}(\epsilon_k) d}{k p_m^2} + \frac{C(X,\coeff^{\star})}{k p_m^3}\right ).$ As a consequence, the variance upper bound is larger by a factor $ p_m^{-1} $ for the estimator derived from $ k $ incomplete observations than for $ k \times p_m $ complete observations. 
	This suggests that there is a higher gain to collecting fewer complete observations (e.g., 100) than more incomplete ones (e.g., 200 with $ p=0.5 $). However, one should keep in mind that this observation is made by comparing upper bounds thus does not necessarily reflect what would happen in practice.

\paragraph{Keeping only complete observations?}
Another approach to solve the missing data problem is to discard all observations that have at least one missing feature. The probability that one input is complete, under our missing data model is $ \prod_{j=1}^{d} p_j $. In the homogeneous case, the number of complete observations $ k_{co} $ out of a $ k-$sample thus follows a binomial law $ k_{co} \sim \mathcal{B}(k,p^d) $. With only those few observations, the statistical lower bound is $  \frac{\mathrm{Var}(\epsilon_k) d}{k_{co}}$. In expectation, by Jensen inequality, we get that the lower bound on the risk is larger than $  \frac{\mathrm{Var}(\epsilon_k) d}{  k p^d}$. 

{Our strategy thus leads to an \emph{upper-bound} which is typically $ p^{d-3} $ times smaller than the \emph{lower bound} on the error of any algorithm relying only on complete observations. For a large dimension or a high percentage of missing values, our strategy is thus provably several orders of magnitude smaller than the best possible algorithm that would only rely on complete observations - e.g., if $p=0.9$ and $d=40$, the error of our method is at least 50 times smaller. }

Also note that in Theorem~1 and Lemma~1 in~\citet{ma2017stochastic}, the convergence rate with missing observations suffers from a similar multiplicative factor $ O(p^{-2} + \kappa p^{-3} ) $.

%% file: xp.tex
\vspace{-0.2cm}
\subsection{Synthetic data}
\vspace{-0.2cm}
Consider the following simulation setting:
the covariates are normally distributed, $\design_{i:} \overset{i.i.d.}{\sim} \mathcal{N}(0,\Sigma)$, where $\Sigma$ is constructed using uniform random orthogonal eigenvectors and decreasing eigenvalues $1/k, \: k = 1, \hdots, d$. For a fixed parameter vector $\beta$, the outputs $y_i$ are generated according to the linear model \eqref{eq:linearmodel}, with $\noise_i \sim \mathcal{N}(0,1)$. Setting  $d=10$, we introduce $30\%$ of missing values either with a uniform probability $p$ of missingness for any feature, or with probability $p_j$ for covariate $j$, with $j=1,\hdots, d$. 
Firstly, the three following algorithms are implemented:
\begin{enumerate}[label={(\arabic*)},topsep=0pt]
	\item \textbf{AvSGD}\label{AvSGD} described in Algorithm \ref{alg:SGDav} with a constant step size $\alpha=\frac{1}{2L}$, and $L$ given in \eqref{eq:smooth_const}. 
	\item \textbf{SGD}\label{SGD} from \citep{ma2017stochastic} with iterates
	$\coeff_{k+1}=\coeff_k-\alpha_k\tilde{g}_{i_k}(\beta_k),$
	and decreasing step size $\alpha_k=\frac{1}{\sqrt{k+1}}.$
	\item \textbf{SGD\_cst}\label{SGD_cst} from \citep{ma2017stochastic} with a constant step size $\alpha=\frac{1}{2L}$, where $L$ is given by \eqref{eq:smooth_const}. 
\end{enumerate}

\begin{figure}
\centering
\begin{minipage}{.47\textwidth}
  \centering
  \begin{figure}[H]
	\begin{center}
		\includegraphics[height=0.49\textwidth]{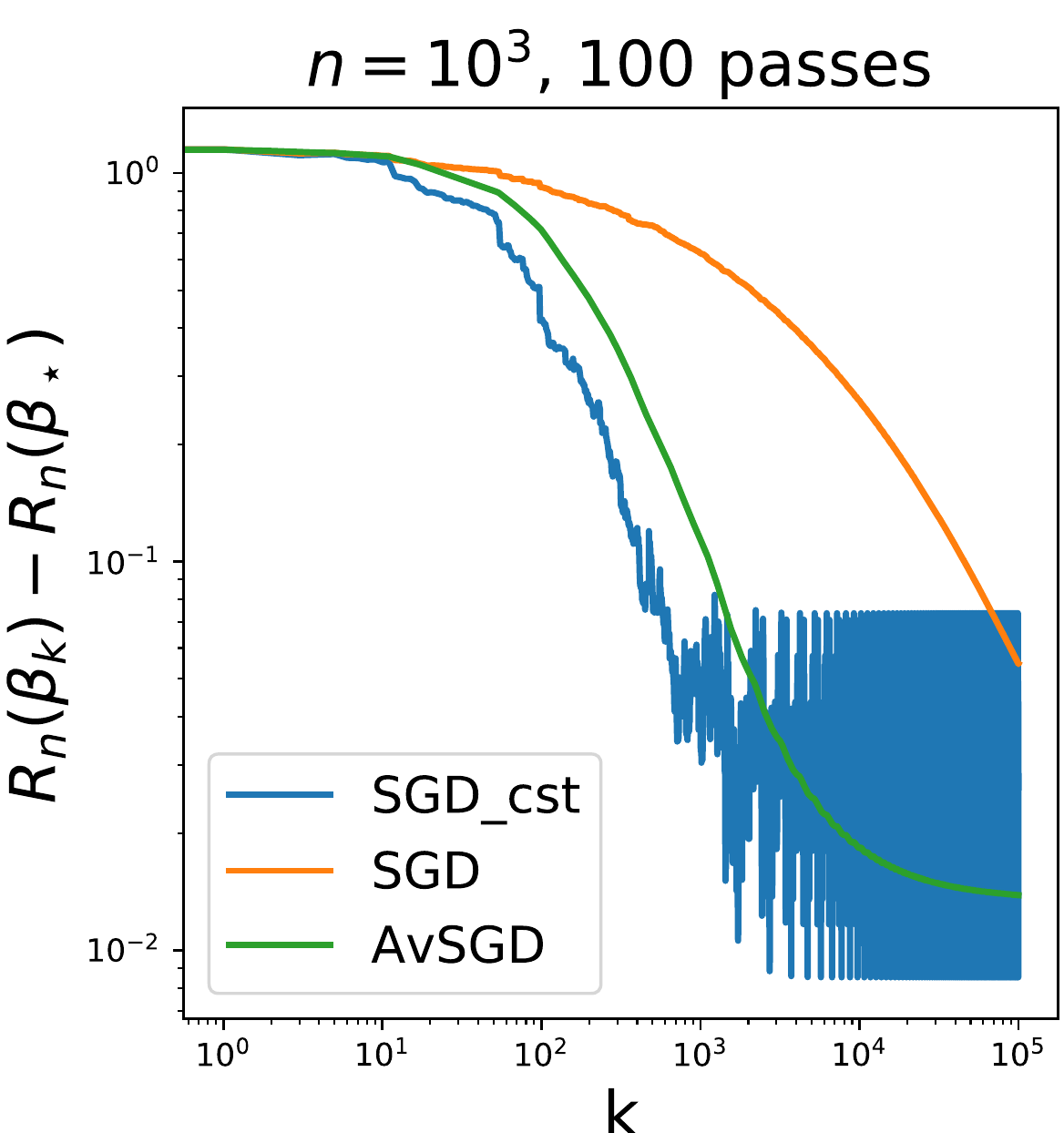}
		\includegraphics[height=0.49\textwidth]{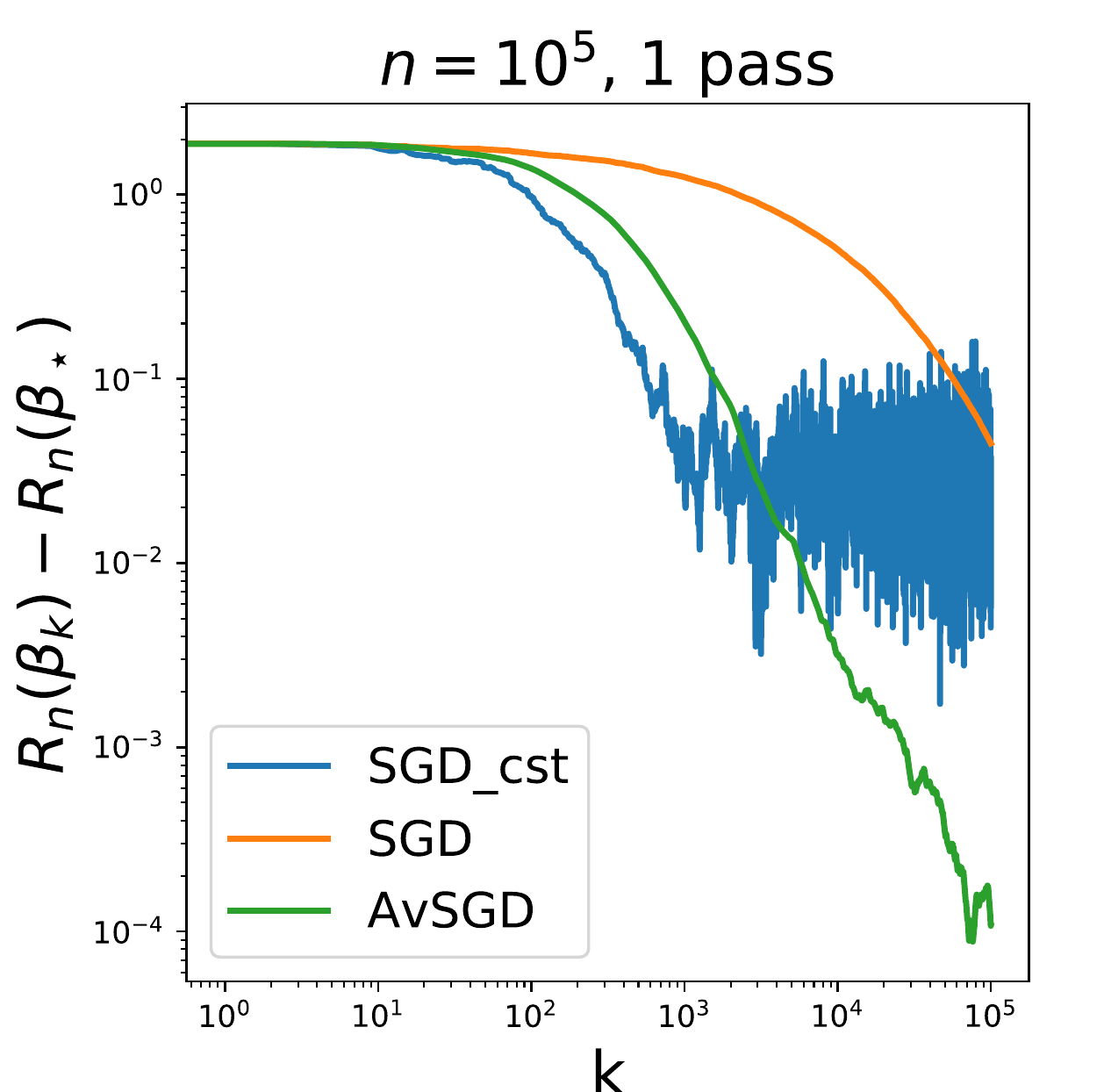}
		\caption{ \label{fig:noseveralpasses} Empirical excess risk $\left(R_\dimob(\coeff_k)-R_\dimob(\coeff^{\star})\right)$. Left: $n=10^3$ and 100 passes. Right:  $n=10^5$ and 1 pass.  $d=10$, 30\% MCAR data. $L$ is assumed to be known in both graphics.}
	\end{center}
  \end{figure}
\end{minipage}%
\hspace{0.04\textwidth}
\begin{minipage}{.47\textwidth}
  \centering
	\begin{center}
		\includegraphics[width=\textwidth]{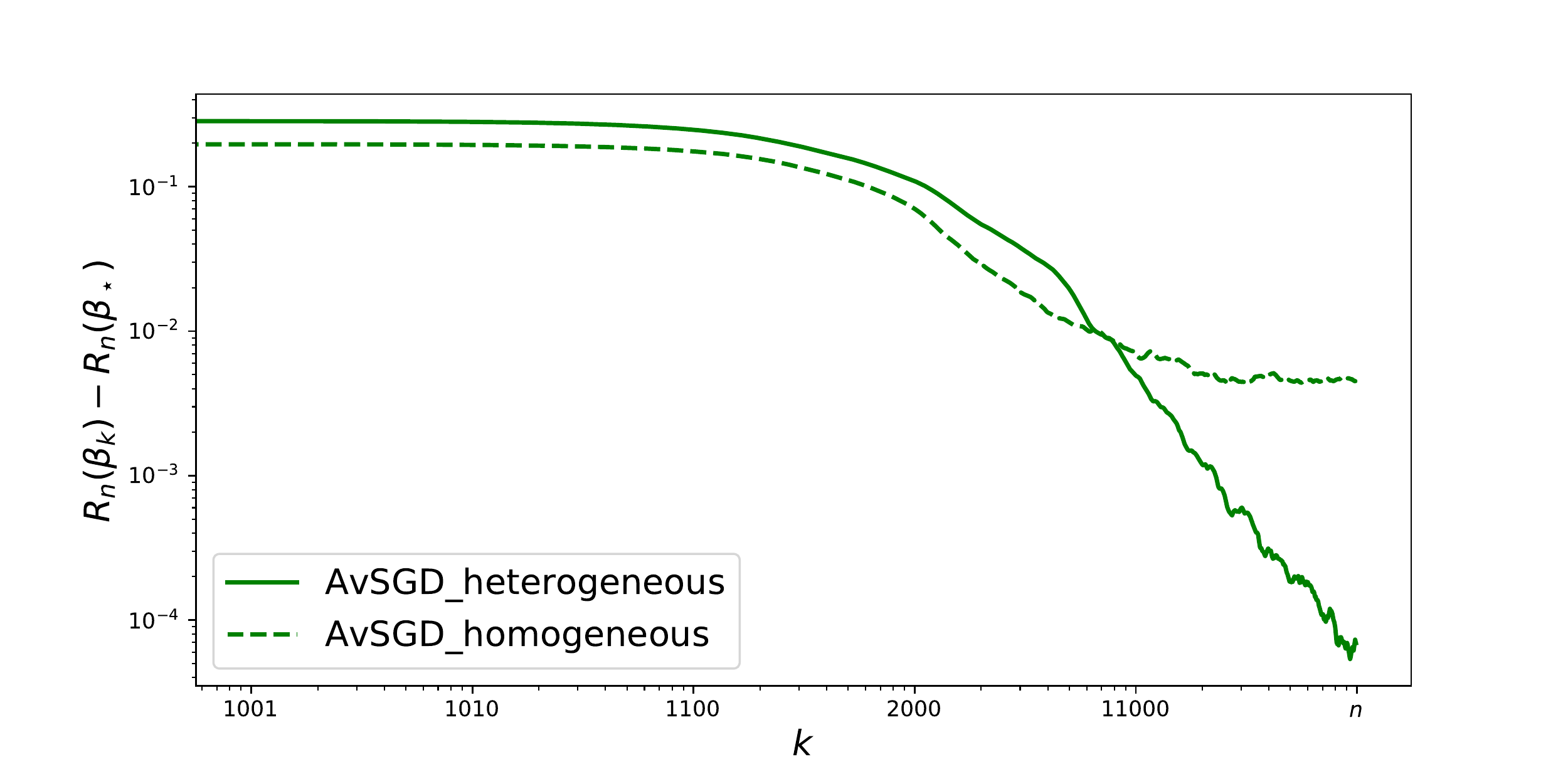}
		\caption{ \label{fig:manyprob}  Empirical excess risk $R_\dimob(\coeff_k)-R_\dimob(\coeff^{\star})$  for synthetic data where $n=10^5$, $d=10$ and with heterogeneous missing values either taking into account the heterogeneity (plain line) in the algorithm or not (dashed line).}
	\end{center}
\end{minipage}
\end{figure}

\paragraph{Debiased averaged vs. standard SGD.} \Cref{fig:noseveralpasses} compares the convergence  of Algorithms \ref{AvSGD}, \ref{SGD} and \ref{SGD_cst}, with either multiple passes or one pass, in terms of excess empirical risk $R_\dimob(\coeff)- R(\beta^{\star})$, with $R_\dimob(\coeff)  := \frac{1}{\dimob} \sum_{i=1}^n f_i(\beta).$
As expected (see \Cref{rem:noseveralpasses,rem:ERM}), multiple passes can lead to saturation: after one pass on the observations, AvSGD does not improve anymore (\Cref{fig:noseveralpasses}, left), while it keeps decreasing in the streaming setting  (\Cref{fig:noseveralpasses}, right). Looking at \Cref{fig:noseveralpasses} (right),  one may notice that
without averaging and with decaying step-size, Algorithm~\ref{SGD} achieves the convergence rate 
$\mathcal{O}\left(\sqrt{\frac{1}{n}}\right)$, 
whereas with constant step-size, Algorithm~\ref{SGD_cst} saturates at an excess risk proportional to $\alpha$ after  $n=10^3$ iterations. As theoretically expected, both methods are improved with averaging. Indeed, Algorithm~\ref{alg:SGDav} converges pointwise with a rate of $\mathcal{O}(\frac{1}{n})$.
\vspace{-0.3em}
\paragraph{About the algorithm hyperparameter.} Note that the Lipschitz constant $L$ given in \eqref{eq:smooth_const} can be either computed from the complete covariates,
or estimated from the incomplete data, see discussion and numerical experiments in \Cref{appsec:Additional-figures}.
\vspace{-0.3em}
\paragraph{Heterogeneous vs. homogeneous missingness.} In Figure \ref{fig:manyprob}, the missing values are introduced with different missingness probabilities, i.e.\ with distinct $(p_j)_{1\leq j \leq d}$ per feature, as described in  \Cref{eq:binarymaskprob}. When taking into account this heterogeneousness, Algorithm~\ref{alg:SGDav} achieves the same convergence rates as in \Cref{fig:noseveralpasses}.
However,  ignoring the heterogeneous probabilities in the gradient debiasing leads to stagnation far from the optimum in terms of empirical excess risk.
\vspace{-0.3em}
\paragraph{Polynomial features.} Algorithm \ref{alg:SGDav} can be adapted to handle missing polynomial features, see \Cref{appsec:poly_features} for a detailed discussion and numerical experiments on synthetic data. 



\subsection{Real dataset 1: Traumabase dataset}
\vspace{-0.5em}
We illustrate our approach on a public health application
with the APHP TraumaBase$^{\mbox{\normalsize{\textregistered}}}$ Group (Assistance Publique - Hopitaux de Paris) on the management of traumatized patients. Our aim is to model the level of platelet upon arrival at the hospital from the clinical data of 15785 patients. The platelet is a cellular agent responsible for clot formation and it is essential to control its levels to prevent blood loss and to decide on the most suitable treatment. 
A better understanding of the impact of the different features  is key to  trauma management.
Explanatory variables for the level of platelet consist in seven quantitative (missing) variables, which have been selected by doctors. In \Cref{tab:signcoef},
one can see the percentage of missing values in each variable, varying from 0 to 16\%, see  \Cref{appsec:Traumabase} for more information on the data.

\vspace{-0.3em}
\paragraph{Model estimation.}
The model parameter estimation is performed either using the AvSGD Algorithm \ref{alg:SGDav} or an Expectation Maximization  (EM) algorithm \cite{dempster1977maximum}. Both methods are compared with the ordinary least squares linear regression in the complete case, i.e. keeping the fully-observed rows only (i.e.\ 9448 rows).
The signs of the coefficients for Algorithm \ref{alg:SGDav} are shown in Table \ref{tab:signcoef}. 

\begin{figure}
\centering
\begin{minipage}{.35\textwidth}
	\centering
\resizebox{0.9\linewidth}{!}{	
    \begin{tabular}{llr}
		Variable & Effect & NA \%  \\ \hline
		Lactate   & $-$  & 16\%\\ 
		$\Delta$.Hemo & $+$ & 16\%\\ 
		VE         &  $-$& 9\% \\ 
		RBC        & $-$ & 8\% \\ 
		SI         & $-$ & 2\%\\ 
		HR         & $+$ & 1\% \\ 
		Age        & $-$ & 0\% \\ \hline
    \end{tabular}}
\caption{\label{tab:signcoef}Percentage of missing features, and effect of the variables on the platelet for the TraumaBase data when the AvSGD algorithm is used. ``$+$'' indicates positive effect while  ``$-$'' negative. }
\end{minipage}%
\hspace{0.05\textwidth}
\begin{minipage}[t]{.59\textwidth}
	\begin{center}
		\includegraphics[width=0.65\textwidth]{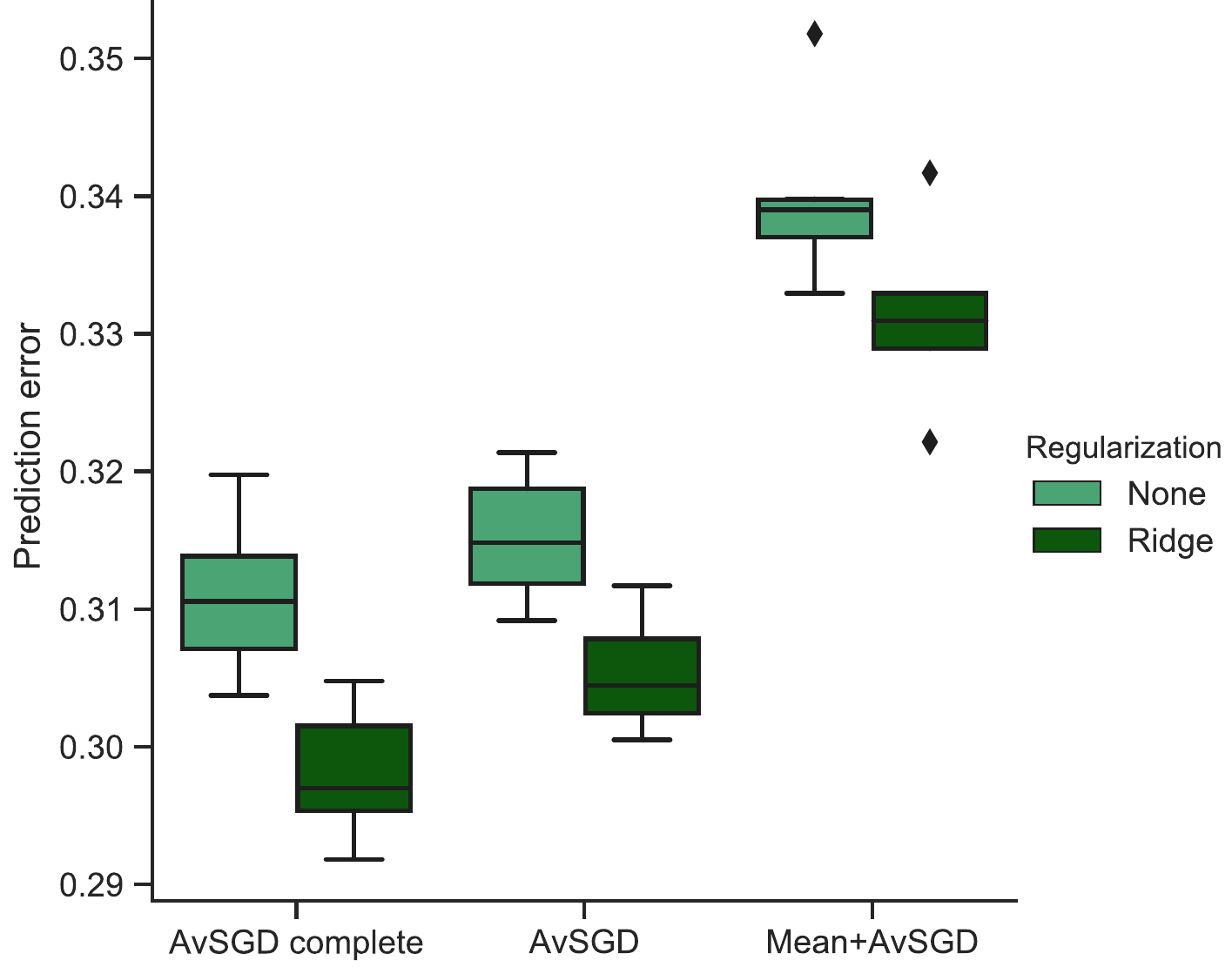}
		\caption{ \label{fig:prediction_supra0}  Prediction error boxplots (over 10 replications) for the Superconductivity data.  AvSGD complete corresponds to applying the AvSGD on the complete data, 
		AvSGD and Mean+AvSGD use the predictions obtained with the estimated parameters $\hat \beta_n^{\textrm{AvSGD}}$ and $\bar{\beta}_n^{\textrm{AvSGD}}$ respectively.}
	\end{center}
	\vspace{-2em}
\end{minipage}
\end{figure}

According to the doctors, a negative effect of shock index (\textit{SI}), vascular filling (\textit{VE}), blood transfusion (\textit{RBC}) and lactate (\textit{Lactacte}) was expected, as they all result in low platelet levels and therefore a higher risk of severe bleeding. However, the effects of delta Hemocue (\textit{Delta.Hemocue}) and the heart rate (\textit{HR}) on platelets are not entirely in agreement with their opinion.
Note that using the linear regression in the complete case and  the EM algorithm lead to the same sign for the variables effects as presented in Table \ref{tab:signcoef}.



\subsection{Real dataset 2: Superconductivity dataset}
\label{subsubsec:superconductivity}
\vspace{-0.5em}
We now consider the Superconductivity dataset (available  \href{https://archive.ics.uci.edu/ml/datasets/Superconductivty+Data}{here}), which contains 81 quantitative features from 21263 superconductors. The goal here is to predict the critical temperature of each superconductor. Since the dataset is initially complete, we introduce $30\%$ of missing values with probabilities $(p_j)_{1\leq j \leq 81}$ for the covariate $j$, {with $p_j$ varying between 0.7 and 1.}
The results are shown in Figure \ref{fig:prediction_supra0} where a Ridge regularization has been added or not. The regularization parameter $\lambda$ (see Remark \ref{rem:ridge}) is chosen by cross validation. 
\vspace{-0.3em}
\paragraph{Prediction performance.}
The dataset is divided into training and test sets (random selection of $70-30\%$). The test set does not contain missing values. 
In order to predict the critical temperature of each superconductor, we compute $\hat{y}_{n+1} = X_{n+1}^T \hat{\beta}$ with $\hat{\beta} = {\beta}_n^\textrm{AvSGD}$ or $\beta_n^\textrm{EM}$. We also impute the missing data naively by the mean in the training set, and apply the averaged stochastic gradient without missing data on this imputed dataset, giving a coefficient model $\bar{\beta}_n^{\textrm{AvSGD}}$. It corresponds to the case where the bias of the imputation has not been corrected.
The prediction quality on the test set is compared according to the relative $\ell_2$ prediction error, $\|\hat{y}-y\|^2/\|y\|^2$. The data is scaled, so that the naive prediction by the mean of the outcome variable leads to a prediction error equal to 1. In Figure \ref{fig:prediction_supra0}, we observe that the SGD strategies give quite good prediction performances.
The EM algorithm is not represented since it is completely out of range (the mean of its prediction error is 0.7), which indicates that it struggles with a large number of covariates. As for the AvSGD Algorithm, it  performs well in this setting. Indeed, with or without regularization, the prediction error  with missing values is very close to the one obtained from the complete dataset. {Note that \Cref{alg:SGDav} is shown to handle missing polynomial features well even in higher dimensions, see \Cref{appsec:poly_features} for a detailed discussion and large-scale experiments on the superconductivity dataset.}

\vspace{-0.3cm}
\section{Discussion}
 \vspace{-0.3cm}
In this work, we thoroughly study  the impact of missing values for Stochastic Gradient Descent algorithm {for Least Squares Regression. We leverage both the power of averaging and a simple and powerful debiasing approach to derive  tight and rigorous convergence guarantees for the generalization risk of the algorithm.}
The theoretical study directly translates into practical recommendations for the users  and a byproduct is the availability of a python implementation of regularized regression with missing values for large scale data, which was not available. Even though we have knocked down some barriers, there are still exciting perspectives to be explored as the robustness of the approach to rarely-occurring covariates, or  dealing with more general loss functions {as well} - for which it is challenging to build a debiased gradient estimator from observations with missing values, or also considering more complex missing-data patterns such as missing-not-at-random mechanisms.


